\documentclass[a4paper,11pt,reqno]{article}
\usepackage[margin=1.2in]{geometry}

\usepackage{color, colortbl}

\usepackage{amsmath}
\usepackage{amsfonts}
\usepackage{amsthm}
\usepackage{amssymb}
\usepackage{tikz}
\usepackage[all]{xy}
\usepackage{tikz-cd}
\usepackage[colorlinks]{hyperref}
\usepackage{enumerate}
\usepackage{mathtools}
\usepackage{setspace}
\usepackage{tocloft}

\usepackage{cancel}

\usepackage{comment}

\definecolor{lightgray}{gray}{0.9}
\definecolor{darkgreen}{rgb}{0,0.5,0}
\definecolor{darkblue}{rgb}{0,0.1,0.5}

\hypersetup{
    colorlinks=true, 
    linkcolor=darkblue, 
    urlcolor=blue, 
    linktoc=all, 
    citecolor=darkgreen
    }

\newcommand{\doi[2]}{\href{http://dx.doi.org/#1}{#2}}
\usepackage{thmtools}
\usepackage{mathtools}

\newtheoremstyle{introTheorems}
  {\topsep}
  {\topsep}
  {\itshape}
  {0pt}
  {\bfseries}
  {}
  { }
  {\thmname{\bf #1}
  \textnormal{\thmnote{#3}.\!}
  }

\newtheorem{theorem}{Theorem}[section]

\newtheorem{corollary}[theorem]{Corollary}
\newtheorem{definition}[theorem]{Definition}
\newtheorem{acknowledgment}[theorem]{Acknowledgment}
\newtheorem{example}[theorem]{Example}
\newtheorem{lemma}[theorem]{Lemma}
\newtheorem{problem}[theorem]{Problem}
\newtheorem{question}[theorem]{Question}
\newtheorem{remark}[theorem]{Remark}
\theoremstyle{introTheorems}
\newtheorem{introTheorem}{Theorem}
\newtheorem{introLemma}{Lemma}

\usepackage[abbrev,msc-links, alphabetic, backrefs]{amsrefs}

\BibSpec{arXiv}{%
  +{}{\PrintAuthors}{author}
  +{,}{ \textit}{title}
  +{}{ \parenthesize}{date}
  +{,}{ arXiv }{eprint}
}

\newcommand{\Rep}{{\rm Rep}}

\newcommand{\cC}{\mathcal{C}}

\newcommand{\cM}{\mathcal{M}}

\newcommand{\V}{\mathcal{V}}

\newcommand{\C}{\mathbb{C}}

\newcommand{\Z}{\mathbb{Z}}
\newcommand{\N}{\mathbb{N}}

\newcommand{\g}{\mathfrak{g}}
\renewcommand{\sl}{\mathfrak{sl}}

\newcommand{\SL}{\mathrm{SL}}

\newcommand{\Vir}{\mathrm{Vir}}

\newcommand{\bpsi}{\bar{\psi}}

\renewcommand{\d}{\mathrm{d}}
\renewcommand{\i}{\mathrm{i}}

\newcommand{\bV}{\mathbb{V}} 

\renewcommand{\Re}{\mathfrak{Re}}

\newcommand{\Res}{\mathrm{Res}}

\newcommand{\gr}{\mathrm{gr}}

\newcommand{\SF}{\mathcal{SF}}
\newcommand{\Ind}{\mathrm{Ind}}
\newcommand{\TFT}{\mathcal{Z}}
\newcommand{\connection}{\nabla_z}

\newcommand{\normord}[1]{\vcentcolon\mathrel{#1}\vcentcolon}
\providecommand{\vcentcolon}{\mathrel{\mathop{:}}}
\newcommand{\indic}{1}
\newcommand{\Shift}{\mathrm{Shift}}

\title{Twisted vertex algebra modules for irregular connections: \\ A case study}
\author{Boris L. Feigin, Simon D. Lentner}
\date{November 2024}

\begin{document}

\maketitle

\begin{abstract}
A vertex algebra with an action of a group $G$ comes with a notion of $g$-twisted modules, forming a $G$-crossed braided tensor category. For a Lie group $G$, one might instead wish for a notion of $(\d+A)$-twisted modules for any $\g$-connection on the formal punctured disc. For connections with a regular singularity, this reduces to $g$-twisted modules, where $g$ is the monodromy around the puncture. The case of an irregular singularity is much richer and involved, and we are not aware that it has appeared in vertex algebra language. The present article is intended to spark such a treatment, by providing a list of expectations and an explicit worked-through example with interesting applications.  

\medskip

Concretely, we consider the vertex super algebra of symplectic fermions, or equivalently the triplet vertex algebra $\mathcal{W}_p(\sl_2)$ for $p=2$, and study its twisted module with respect to irregular $\sl_2$-connections. We first determine the category of representations, depending on the formal type of the connection. Then we prove that a Sugawara type construction gives a Virasoro action and we prove that as Virasoro modules our representations are  direct sums of Whittaker modules.

\medskip

Conformal field theory with irregular singularities resp. wild ramification appear  in the context of geometric Langlands correspondence, in particular work by Witten \cite{Wit08}, and more generally in higher-dimensional context. Our original motivation comes from semiclassical limits of the generalized quantum Langlands kernel, which fibres over the space of connections \cite{FL24}, similar to the affine Lie algebra at critical level. Our present article now describes, in the smallest case, the fibres and their representation categories over irregular connections.
\end{abstract}

\newpage

\tableofcontents


\section{Introduction}

\begin{quote}
Why should an algebraic geometer be interested in
irregular connections?\\
(Pierre Deligne)
\end{quote}

\subsection{Twisted modules, the global picture}

Before we come to the concrete results in the present article, we discuss the intuition in the background. We expect that the following view is shared implicitly by the experts, but we have not found a rigorous treatment, to the point we need, in vertex algebra literature. Hence, in this paper we will give ad-hoc definitions and subsequent results on the vertex algebra side, and the interpretation that we now give should be considered conjectural. It would be good to clarify them in future work and we are happy for comments or references.

Let $\V$ be a vertex operator algebra with good properties, for example $C_2$-cofinite. An important global quantity  associated to $\V$ are the \emph{chiral conformal blocks} $\TFT(\Sigma_{g,n})$  associated to any Riemann surface with genus $g$ (including a complex structure) and $n$~punctures $z_1,\ldots,z_n$ \cite{FBZ04}. The axioms require a good dependence on $z_1,\ldots,z_n$ and a good behavior under gluing of surfaces. The vertex operator itself, the action on vertex modules and most generally the intertwiners between triples of vertex modules correspond to elements in the chiral conformal block of the three-punctured sphere $\TFT(\Sigma_{0,3})$.

Assume that a Lie group $G$ acts on $\V$ by vertex algebra automorphisms.  We would expect the notion of a \emph{$(\d+A)$-twisted module} $\cM$ where 
$$\d+A=\d+\sum_{k\in\Z} A_n z^{-n-1},\quad A_n\in \g$$
is a formal $\g$-connection on the punctured formal disc. Of course this should have good transformation properties under formal coordinate transformations and under formal gauge transformations, which are formal power series $F\in G((z))$ and which change the connection to 
$$F[\d+A]=\d-(\d F)F^{-1}+FAF^{-1}$$
If $\cC$ is the category of representations of $\V$, then we denote by $\cC_{\d+A}$ the category of $(\d+A)$-twisted modules. Precomposing with a gauge transformation should send a $(\d+A)$-twisted module to a $F[\d+A]$-twisted module.\footnote{We check this in our example in Lemma \ref{lm_gaugeModules}.}

We moreover expect that for every global $\g$-connection $\connection$ on the sphere with punctures $z_1,z_2,z_3$ there is a bifunctor
\begin{align}\label{formula_bifunctor}
\cC_{\d+A^{(1)}}\boxtimes \cC_{\d+A^{(2)}} \to \cC_{\d+A^{(3)}}
\end{align}
where $\d+A^{(i)}$ is the expansion of $\connection$ around $z_i$. This (and any other global quantity) should have good transformation properties under  coordinate transformations and analytic gauge transformations i.e. the power series above have to be convergent.

\begin{example}
In the case of a connection with regular singularity $\d+\sum_{n\leq 0} A_n z^{-n-1}$, the connection on the formal disc is up to gauge transformations classified by the monodromy $g=e^{2\pi\i A_0}$. We may denote $\cC_{\d+A}=:\cC_g$ up to gauge transformations. Note that by going to these equivalence  classes we disregard possible higher categorical structure.

The space of global connections on the $3$-punctured spheres with simple poles of monodromy $g_1,g_2,g_3$ is, up to gauge transformations, $1$-dimensional iff $g_1g_2g_3=1$ respectively $g_1g_2=g_3$ (for ingoing orientations at $z_1,z_2$ and outgoing orientation at $z_3$) and zero else. The bifunctor \eqref{formula_bifunctor} is expected to turn $\bigoplus_{g\in G}\cC_g$ into a $G$-crossed braided tensor category.  
\end{example}

\begin{problem}
These definitions should be brought into concrete vertex algebra definitions, that reduce in the case of a regular singularity to the familiar notions of $g$-twisted modules \cites{FLM88, Huang10}. 
\begin{itemize}
\item The $(\d+A)$-twisted modules should be defined by  a \emph{$(\d+A)$-twisted Jacobi identity}. Starting from the commutator formula in \cite{Bak16} Section 5.3, to which the reader is referred for details, a guess for a generalization to $(\d+A)$-twisted modules is\footnote{We check this in our example in Lemma \ref{lm_twistedCommutator}.}
\begin{align}\label{formula_twistedJacobi} 
[a_{(m)}^{\d+A},b_{(n)}^{\d+A}]=\sum_{j\geq 0}\sum_{k\in\Z} \big(({m+A_k\choose j} a)_{(j)} b\big)_{(m+n-j-k)}^{\d+A}
\end{align}
where we plug $A_k\in\g$, which acts on fields, into the complex polynomial ${m+X\choose j}$. The regular singular term $A_0$ preserves the degree, while the regular terms contribute with strictly lower degree and the irregular terms contribute with strictly higher degree.  

\item The fixed-point vertex subalgebra $\V^G\subset \V$ should act as for usual vertex modules. In particular, the  $(\d+A)$-twisted modules should be proper modules over the Virasoro algebra.\footnote{We check this in our example in Lemma \ref{lm_Vir}.} 
\item A bifunctor \eqref{formula_bifunctor} depending on the global connection $\connection$ should be defined by a notion of $\connection$-twisted intertwining operators.\footnote{Note that even for regular singularities and non-commuting group elements $g_1,g_2,g_3$ there seems not yet a satisfactory definition of such intertwining operators in terms of Jacobi identities. }
\end{itemize}
Moreover, if the $G$-action is \emph{inner}, that is, it comes from a conformal embedding of the vertex algebra associated to the affine Lie algebra $\bV^\kappa(\g)\hookrightarrow \V$, as in \cite{FBZ04} Section 7.    
Then there should be a $\Delta$-deformation $\cC\cong \cC_{\d+A}$  generalizing the respective construction in  \cites{DLM96,Li97, AM09}.
\end{problem}

\begin{example}\label{exm_affine}
For $\V=\bV^\kappa(\g)$ itself with the inner action, we expect that $\cC\cong \cC_{\d+A}$ by an explicit isomorphism of the twisted and untwisted  mode algebra as in \cite{FL24} Lemma~2.2. However, the Virasoro action is modified, as we have encountered in \cite{FL24} Theorem 5.3
\begin{align*}
L_n^{\d+A}
&=L_n
+\sum_{n'+n''=n} (A_{n''})_{n'}
-\frac{1}{2}\sum_{n'+n''=n}  \langle A_{n'},A_{n''}\rangle,
\end{align*}
for $\kappa=1$, where the notation is $A_n\in\g$ and $(A_n)_m\in \hat{\g}$.
\end{example}
\begin{example}\label{exm_SF}
Let $\V=\SF$ the vertex super algebra of symplectic fermions resp. its even part, which carries an action of $\mathrm{SL}_2$, which we silently identify with $\mathrm{Sp}_2$. It is the smallest case of the Feigin-Tipunin algebra \cites{FGST05,FT10, AM14, Len21, Sug21, CLR23} and its category of representations is equivalent to the category of representations of a small (quasi-)quantum group $\tilde{u}_q(\sl_2)$ for $q^4=1$, as was proven in \cites{FGST05, AM08, TW13, CGR20, CLR21, GN21, CLR23}. Then the algebra $\SF_{\d+A}$ defined by the concrete commutator relations \cite{FL24} Lemma 7.3 and below in Definition \ref{def_deformedModeAlgebra} should correspond to $(\d+A)$-twisted modules.
\end{example}

In the present article, we study the case of irregular singularities, that is, $A(z)$ has a pole of order $N+1$, where $N\geq 1$ is called  \emph{Poisson order}, so locally
$$\d+A=\d+\sum_{n\leq N} A_n z^{-n-1}$$
with $A_0$ the regular singular term and $A_1,\ldots,A_N$ irregular terms. The most tangible implication is that, for irregular singularities, the  Frobenius method of solving the corresponding differential equation first for $A_0$ and then recursively by a power series fails, because the powers grow in both positive and negative power directions. 

We assume now for simplicity the Poisson order to be $N=1$. We further assume for simplicity that the irregular term $A_1$ is a  semisimple element of $\g$, which is a usual nondegeneracy assumption. Then for sufficiently small $z$ the matrix-valued function $A(z)$ is diagonalizable and thus admits an analytic similarity transformation to 
$$\d+\Xi z^{-2}+ \Lambda z^{-1}+\text{ regular terms }$$
with $(\Xi,\Lambda)$ two Lie algebra elements in the Cartan part $\mathfrak{h}\subset \g$ (for $\sl_n$: diagonal matrices). 
One can obtain a formal gauge transformation $F(z)\in G((z))$, such that the connection is in \emph{Birkhoff normal form} \cite{Birk1913a}
$$\tilde{A}=\Xi z^{-2}+ \Lambda z^{-1},\qquad
F[\d+\tilde{A}]=\d+A$$
see for example \cite{JLP76a} Section 3 or \cite{Boalch02} Lemma 1. The invariants $(\Xi,\Lambda)$ are called \emph{formal type} of $\d+A$. 
The results in the present article on $\cC_{\d+A}$ will, as expected, turn out to only depend on the formal type $(\Xi,\Lambda)$.

In the context of conformal field theory, irregular singularities appear in work of Witten \cite{Wit08}, who aims at extending the gauge theory approach to geometric Langlands correspondence to wild ramification, in work of Feigin, Frenkel and Toledano Laredo \cite{FFTL10} for Gaidin modules and in work of Gaiotto and Teschner \cite{GT12} studying irregular conformal blocks of the Virasoro algebra motivated by AGT-correspondence for Argyres-Douglas theories.Needless to say, the concept of a category fibring over the stack of connections resp. local systems, including irregular singularities resp. wild ramifications, is abundantly present in the topic of Langland correspondence, see for example \cite{Gait15}. Perhaps some aspects below  become clearer in the context of Moy-Prasad theory \cite{Yang22}.
Regarding the deformation by a connection in general, it is interesting to compare our setup in \cite{FL24} with the $3$D quantum field theories and  $4$D-topological field theories obtained as well from the big center quantum group in \cites{CDGG24,Kin24}. 

\begin{remark}
Associated to a connection is the differential equation 
$$(\d+A(z))\psi(z)=0.$$
It appears in several ways in our context: Its formal solutions can be used to transform the connection to the trivial connection, formally. In the explicit setting of this article, it could be used to determine infinite linear combinations of $\psi^\pm_i$ which fulfill the relations of the untwisted symplectic fermions, and in the same way they give for the irreducible twisted modules below (which becomes indecomposable for $A=0$, see Example \ref{exm_regular}) submodules in some algebraic closure.
On the other hand, Lemma \ref{lm_trafoGlobal} links the differential equation to the action of $L_{-1}$, which is not surprising in view of the Ward identities. 
\end{remark}

We finally discuss the \emph{Stokes phenomenon}: Typically, in the irregular case the formal series $F(z)$ is not convergent in any neighborhood of $0$. Hence, in order to classify \text{$\g$-connections} up to analytic gauge transformations, we need more invariants in addition to the Birkhoff normal form \cite{Birk1913b}.

This turns out to be connected to the observation that solutions of differential equations with irregular singularity have essential singularities in $z=0$ and those have different asymptotics in different sectors, separated by directions where the solutions oscillates. A standard example is the Airy-equation at $z=\infty$. Let us instead give a trivial example:
\begin{example}
The differential equation 
$$\left(\frac{\d}{\d z}+\begin{pmatrix} \lambda & 0 \\0 & -\lambda \end{pmatrix}z^{-2}\right)\psi(z)=0$$
has a basis of solutions 
$$\begin{pmatrix} e^{\lambda/z} \\ 0 \end{pmatrix},\quad
\begin{pmatrix} 0 \\ e^{-\lambda/z}  \end{pmatrix}
$$
The solutions have different asymptotics in the two sectors $\Re(z)>0$ and $\Re(z)<0$, and they are oscillating on the rays $\Re(z)=0$.
\end{example}
 We briefly discuss the case $\g=\sl_2$ and Poisson order $1$, which has been worked out in great detail in \cites{JLP76a,JLP76b}, and for general $\g$ and general order we refer the reader to \cite{Boalch02} Section 2 and  \cite{Boalch14} Section 7, which we now discuss briefly: The single pair of anti-Stokes directions are given by the eigenvalues $\pm\lambda$ of the irregular term $A_1$, in these directions the solution decreases most rapidly. They are representatives for the two sectors (directions with positive and negative real value of $\lambda/z$) in which one of the solutions has a good asymptotic $z\to 0$. Now, to the irregular term is associated its stabilizer subgroup $H$, in our case generically only the torus, and for each direction $d$ a Stokes subgroup, in our case the positive and negative unipotent subgroup $U_\pm$. To each $\g$-connection $\d+A$ with irregular term $\Xi$ is associated an invariant in the Poisson dual group of $\mathrm{SL}_2$ by
$$\d+A\quad \longmapsto \quad (e^{2\pi\i\Lambda}, S_+,S_-) \in H\times U_+\times U_-, $$
where the Stokes matrices $S_\pm$ roughly measure the transformation between of the solutions with good asymptotics in different Stokes sectors. This is a complete set of invariants of $\sl_2$-connections of Poisson order~$1$ up to analytic gauge transformations. Moreover, analytic solutions can be obtained from the formal solutions in a given sector by Borel resummation in the direction $d$. In our example involving $e^{\pm\lambda/z}$, the Stokes matrices are trivial, or differently said, the solutions with good asymptotics and the Borel resummations in different Stokes sectors are consistent. However, a similar differential equation with an additional non-diagonal term is formally, but not analytically, equivalent to this, and has nontrivial Stokes matrices. We discuss this example (whose solution is the incomplete Gamma function, and as special cases the exponential integral) in the appendix.  A good discussion with examples can also be found in the appendix of \cite{Wit08}.

\begin{example}[Confluence]\label{exm_confluence}
Irregular singularities arise naturally, if for a parametrized family of connections with regular singularities two of these singularities collide for a limit of the parameter. For example, the hypergeometric functions (with three regular singular points) have as limit the confluent hypergeometric functions (with one irregular and one regular singular point), including the Bessel functions, the Airy functions, and the incomplete Gamma functions we treat in the appendix,  and the Whittaker functions appearing in Whittaker modules of $\mathrm{SL}_2(\mathbb{R})$, which should also be related to our findings. 

For such a collision of regular singularities, the resulting monodromy is clearly the monodromy around both regular singularities, but the individual monodromies remain visible in the Stokes sectors. A nice overview of this can be found in \cite{Hor20}. In \cite{GT12} confluence is used to compute conformal blocks of three-punctured spheres with irregular singularities. 
\end{example}
\medskip
\begin{problem}\label{prob_confluence}
The intertwining operators should be  used to define and determine the bifunctor \eqref{formula_bifunctor}. Does the confluence procedure in \cite{GT12} produce all intertwining operators admitted by Stokes data?

On the other hand, any such intertwining operator should be a usual intertwining operator for the Virasoro algebra. In Examples \ref{exm_affine} and \ref{exm_SF} discussed in this paper, these are Whittaker modules, so one should compute the fusion rules for such modules over the Virasoro algebra, which to our knowledge has not been done, and match the result. 
\end{problem}

P. Boalch has developed the theory of Stokes data into a \emph{wild Riemann-Hilbert correspondence} \cite{Boalch14}, in that he determines explicitly for $\Sigma_{g,n}$ with irregular singularities, under which conditions on the Stokes data in each point there exists a global $\g$-connection with these prescribed singularities. 
\begin{problem}
Develop a notion of crossed braided tensor categories graded by Stokes data. We would very much like to have classification results as in \cites{ENOM10,DN21}. The necessary (higher categorical) notion of a category fibring over a stack of local systems can make use of \cite{Gait15}.
\end{problem}

In the realm of the present article, it is interesting if and where Stokes data will appear. One would expect this, whenever we ask for analytical functions, for example in the definition of fusion rules. Let us also add an interesting problem that is beyond but supposedly related to the present article:
\begin{problem}
The Stokes data carries an action of the braid group, which is mysteriously related to the action of the Hecke algebra on the Kac-DeConcini-Procesi quantum group, who has a large center containing the Poisson dual Lie group $G^*$ \cites{DCKP92, Boalch02, TLX23}. 
On the other hand, we have found in \cite{FL24} regular singularities to correspond to Borel parts of this quantum group.  
\end{problem}

\subsection{Content of this article}

Section \ref{sec_definition} starts with the main definition in this article, an infinite-dimensional Clifford algebra $\SF_{\d+A}$ depending on a connection $\d+A$ on the formal disc. Explicitly, we take generators $\psi^+_n,\psi^-_n$ for all $n\in\Z$ and propose the following relations for the  anticommutators:
$$\{\psi^a_m,\psi^b_n\}=m(e^a,e^b)\delta_{m+n=0}+\sum_{k\in\Z} C_k^{ab}\delta_{m+n=k},$$
for all $a,b\in \{+,-\}$, where $(e^a,e^b)=\begin{psmallmatrix}
0& -1 \\ 1 & 0
\end{psmallmatrix}_{a,b}$ is the standard symplectic form and the symmetric matrix $C_k$ is related to the connection's coefficients $A_k\in\sl_2$  by 
$C_k=\begin{psmallmatrix}
0& -1 \\ 1 & 0
\end{psmallmatrix}A_k$. We also check that this formula would coincide with the proposed commutator formula \eqref{formula_twistedJacobi}.

\begin{example}\label{exm_regular}
For $A=0$ this definition reduces to the mode algebra of the vertex super algebra of symplectic fermions, whose representation theory is related to the small quantum group $\tilde{u}_q(\sl_2)$ at $q^4=1$, with two simple representations and indecomposable projective representations of the typical diamond-shaped form:
\begin{center}
\begin{tikzcd}[scale=0.4]
    & 
    \SF^{<0}v
    \ar[dl,swap,"\psi_0^-", shift right=-0.5ex]
    \ar[dr,"\psi_0^+", shift left=-0.5ex] & \\
    \SF^{<0}\psi_0^-v
    \ar[dr, swap, "\psi_0^+", shift right=-0.5ex] 
    && 
    \SF^{<0}\psi_0^+v 
    \ar[dl, "\psi_0^-", shift left=-0.5ex] \\
    &
    \SF^{<0}\psi_0^-\psi_0^+v
\end{tikzcd} 
\end{center}
For $\d+A$ any regular connection, the category is equivalent to the case $A=0$, as there are only lower degree contributions.
For $\d+A$ any regular singular connection with monodromy $g=e^{2\pi\i A_0}$ in $\mathrm{SL}_2$, there is a category of highest-weight representations equivalent to the $g$-twisted modules of symplectic fermions, depending on $A_0$ being semisimple or nilpotent, see \cite{FL24} Figure 1. 
\end{example}

\begin{remark}
    Representations of $\SF_{\d+A}$ supposedly are examples of $(\d+A)$-twisted representations of the vertex algebra of symplectic fermions in the sense discussed in the previous section.
    
    The explicit definition of $\SF_{\d+A}$, which is the basis of the present paper, we originally obtained in \cite{FL24} Lemma 7.3 as a fibre of a vertex algebra with a big central subalgebra isomorphic to the ring of functions on the space of $\g$-connections. We have constructed this vertex algebra as a semiclassical limit of the generalized quantum Langland kernel, in a similar way as the affine Lie algebra at critical level. 
\end{remark}

We then discuss a global definition of $\SF_{\d+A}$ as algebra of functions with values in the symplectic vector spaces $\C^2$, on which $\SL_2$ acts, turned into a Clifford algebra by  $\Res(\connection f,g)$. This is analogous to the global definition of the affine Lie algebra, see \cite{FBZ04}  19.6. For one, this demonstrates in another way that our initial definition was reasonable, but more importantly, such a global description allows to express higher genus conformal block, and possibly also chiral cohomology \cites{BD04, vEH24}, more directly then solely from the vertex algebra definition.

We finally discuss gauge transformations, how they explicitly look in the local picture, and how they map $(\d+A)$-twisted modules to $F[\d+A]$-twisted modules as expected. We also present an example of a singular gauge transformation, which could be interpreted as spectral flow.

\bigskip

In Section \ref{sec_twistedRepresentations} we discuss a suitable category of representations $\cC$ for $\SF_{\d+A}$ depending on a fixed triangular decomposition and from now on we concentrate on $\d+A$ being irregular of Poisson order $1$ (the case of general Poisson order is similar). The representations resemble Whittaker representations. More precisely, we have a category of modules induced from a given action of the subalgebra $\SF_{\d+A}^{01}$ generated by $\psi_0^\pm,\psi_1^\pm$ (whose relations are nontrivial and depend on the formal type, see Example \ref{ex_PoissonOrderOne}) and with a trivial action of the centralizing subalgebra $\SF_{\d+A}^{\geq 2}$ generated by $\psi_n^\pm,n\geq 2$. We prove that the induction functor is actually an equivalence of abelian categories (similarly to, say, the Heisenberg Lie algebra), so the simple modules are precisely the induced modules of simple modules over $\SF_{\d+A}^{01}$. In the irregular case this algebra turns out to be  always simple, so there is a unique simple module $\smash{\SF_{\d+A}^{\leq 0}}$, on which however $\psi_1^\pm$ acts nonzero. 

There is a \emph{vertical $\Z$-grading} by assigning $\deg_{vert}(\psi_{-i}^\pm)=i$, which is the familiar untwisted $L_0$-grading, and a \emph{horizontal $\Z$-grading} by assigning  $\deg_{horiz}(\psi_{-i}^\pm)=\pm 1$, which is the familiar untwisted $\sl_2$-grading or lattice-grading in a free-field realization, but both gradings need not be preserved in the twisted case.
Regular connections solely produce contributions in strictly lower vertical degree. Regular singular connections also produce contributions in the equal vertical degree and irregular connections also produce contributions in strictly higher vertical degree. If the connection is in Birkhoff normal form, then the horizontal degree is preserved.   

For comparison, we also discuss the more familiar case of a regular singular connection, where there is similarly an induction functor from the subalgebra $\smash{\SF_{\d+A}^{0}}$ generated by $\psi_0^\pm$, and it is also an equivalence of categories. The representation theory of $\SF_{\d+A}^{0}$ depends crucially on the regular singular term $A_0\in \sl_2$ being zero, nilpotent or semisimple, and we encounter accordingly the three different types of twisted modules for the symplectic fermions resp. the Kac-DeConcini-Procesi quantum group in \cite{FL24} Figure 1.

We also remark that the  gauge transformations $\SL_2((z))$ preserve the choice of central subalgebra $\smash{\SF_{\d+A}^{\geq 2}}$ and our subcategories of modules $\cC$, while for example the singular gauge transformation we called spectral flow shifts $\SF_{\d+A}^{\geq 2}$ to a different choice of central subalgebra.

\bigskip

In Section \ref{sec_Virasoro} we turn our attention to the Virasoro algebra. An important, expected, but technically somewhat involved result is a version of the Sugawara construction:

\begin{introLemma}[\ref{lm_Vir}]
The following operators in $\SF_{\d+A}$
$$L_{n}^{\d+A}:=\sum_{i+j=n}\normord{\psi_i^-\psi_j^+}
\,+\,c_n1,\qquad c_n=\frac12\sum_{i+j=n} \left(C_{i}^{-+}C_{j}^{+-}+C_{i}^{++}C_{j}^{--}\right)$$
fulfill the relations of the Virasoro algebra at central charge $c=-2$. 
\end{introLemma}
Here $\normord{\psi_i\psi_j}$ is a particular choice of normal ordering  given in Definition \ref{def_normord}. We remark that the additional constant $c_n$ is the residue of the determinant of $A(z)$ as an $\sl_2$-valued function under the identification. For the usual $\sl_2$-twisted modules, correspondingly $A(z)=A_0z^{-1}$, it recovers the familiar shift in $L_0$ for semisimple elements $A_0$. A~similar shift is the quadratic term in Example \ref{exm_affine}. 

We also prove in Lemma \ref{lm_trafoGlobal} a formula for $[L_n^{\d+A},\psi_i^\pm]$, which could be globally written 
$$[L_{-1}, \psi(z)]=\left(\frac{\partial}{\partial z}+A(z)\right){\psi}(z)$$
if the coefficients $A_n\in\sl_2$ act on $\psi^\pm_k$ with the dual representation $\overline{\C^2}$.

In Section \ref{sec_Formulas} we spell out the action of $\psi_i^\pm$ and $L_n$ on the induced modules very explicitly in terms of (anti-)derivations. We also spell out the considerably simpler  formulas for a connection in Birkhoff normal form. 

As an example, we compute the action of $L_n,\,n>-1$ on the elements $v$ and $\psi_0^-v$. We find in the irregular case that those elements are Whittaker vectors, that is, $L_2$ and $L_1$ act by nonzero scalars. Our example also illustrates, how the indecomposable module in Example \ref{exm_regular} in the untwisted case becomes simple in the irregularly twisted case.

Section \ref{sec_VirasoroStructure} contains our main technical result:
\begin{introTheorem}[\ref{thm_VirasoroStructure}]
For $\d+A$ an $\sl_2$-connection, irregular of Poisson order $1$ as in Example~\ref{ex_PoissonOrderOne}, the simple induced module $\SF_{\d+A}^{\leq 0}$ is as a Virasoro representation the direct sum of Whittaker modules with generator $w_k,k\in\Z$ such that
  \begin{align*}
        L_{n\geq 3}^{\d+A} w_k &=0   \\ 
        L_2^{\d+A} w_k,
        &=\frac12 \xi^2 \cdot w_k \\
        L_1^{\d+A} w_k ,
        &=\xi(k+\frac12 \varepsilon)\cdot w_k.
    \end{align*}
\end{introTheorem}
This result is achieved as follows: First we check in Corollary \ref{cor_WhittakerVectors} that for a connection in Birkhoff normal (which means in particular that the horizontal grading is preserved) the following familiar vectors\footnote{These are of course the top vectors of the Fock modules of the projective representations in the untwisted case in Example \ref{exm_regular}.} 
 \begin{align*}
        v_{k}&=\psi^-_{-(k-1)}\cdots \psi^-_0 \\
        v_0&=1\\
        v_{-k}&=\psi^+_{-(k-1)}\cdots \psi^+_0 
 \end{align*}    
 are Whittaker modules with the asserted parameters, as we have already observed in the previous section in examples. Then, in Corollary \ref{cor_WhittakerVectorsII} we show that vectors $w_k$ with the same property exist without assuming Birkhoff normal form, because the additional terms in $A_i,i\geq 0$ only produce additional terms in lower vertical degree (and possibly different horizontal degree), and the initial choice $v_k$ can be corrected accordingly.

 Next, we discuss equivalence as graded vector spaces, or character formulas: If we compare the direct sum of two Fock modules, which is the untwisted case in Example~\ref{exm_regular}, or similarly  two  Virasoro Verma modules $\C[L_{-1},L_{-2},\cdots]v_k\oplus \C[L_{-1},L_{-2},\cdots]v_{k+1}$ and on the other hand the Virasoro Whittaker module 
$\C[L_0,L_{-1},\cdots]v_k$, then the latter seems much larger. Their equality is made possible by the fact that the vertical grading is not preserved. In a matter of speaking, allowing the degree to increase just by one creates ''infinitely much space''. This becomes transparent if we introduce an additional finer grading on $\SF^{\geq 0}$ given by $\deg_{fine}(\psi_i^\pm)=\frac{1}{2}$. In the untwisted case, we can rewrite the familiar character of Fock modules, which is given by partitions of $n=\deg_{vert}$, in terms of the finer grading $x=\deg_{fine}$. We get a parametrization in terms of pairs of partitions of $n$ into at most $k+x$ resp. $x$ parts, which explicitly correspond to monomials of these degrees in $\psi^-_i$ and $\psi^+_i$. The identity of characters is an easy consequence of $q$-Pascal identity and twice $q$-Vandermonde identity, see Lemma~\ref{lm_characterIdentity}
\begin{align*}
&q^{k(k-1)/2} {M+N-1 \choose M-k}_q
+q^{(k+1)k/2} {M+N-1 \choose M-k-1}_q\\
&=\sum_{x\geq 0} q^{(k+x)(k+x-1)/2} {M \choose k+x}_q
\cdot q^{x(x-1)/2} {N \choose x}_q
\end{align*}

Here, the Gau{\ss} binomial coefficient ${A+B\choose A}_q=\frac{[M+N]_q!}{[M]_q![N]_q!}$ introduced in \cite{Gauß1808} has as coefficient of $q^n$ the number of partitions of $n$ with at most $A$ summand and each summand at most $B$, see for example \cite{And98} Theorem 3.1. The intended identity of graded dimensions follows for $M=N=\infty$, but the finite versions suggest an additional filtration that is useful for inductive proofs, as we shall see.

We now turn to the irregularly twisted module: 

The usual action of a Virasoro algebra element $L_n$ on the Fock module raises $\deg_{vert}$ by $-n$, and the terms of the form $\psi_i^-\psi_j^+$ with $i,j>0$, which we now call $L_n'$, raise in addition $\deg_{finer}$ by one. For the twisted action with an irregular connection, $L_n^{\d+A}$ for $n\leq 0$ has among others a new term we call  $(L_n)_{\partial-1,\partial-1}=\xi\,\Shift_{n-1}$ in Section \ref{sec_Formulas}, which raises $\deg_{vert}$ by $-n+1$, and  there is  unchanged the term $L_n'$ and all other terms raise $\deg_{vert}$ by at most $-n$ and preserve $\deg_{fine}$. For this reason, we set on the induced module
$$\deg_{total}=\deg_{vert}+\deg_{fine}$$
and on $\Vir^{\leq 0}$ we set $\deg_{Vir}(L_n^{\d+A})=-n+1$, then the action of is compatible with these filtration. Note that if we had not added the fine grading, we would only see $\Shift_{n-1}$, which preserves the fine grading and hence generates not nearly the full module. For this new grading, we have an analogous character formula as above, just modified by the fine degree, whose result is the character of the Whittaker modules with shifted degree: 
\begin{align}
\begin{split}
&q^{k(k-1)/2} {M+N \choose M-k}_q\\
&=\sum_{x\geq 0} q^{(k+x)(k+x-1)/2} {M \choose k+x}_q
\cdot q^{x(x-1)/2} {N \choose x}_q \cdot q^x
\end{split}
\end{align}
For later use, we give in Lemma \ref{lm_constructive} a proof of this formula by an explicit bijection of partitions, which is surely known.  
By using the known irreducibility of the universal Whittaker modules, see Theorem~\ref{thm_WhittakerIrrep}, we conclude injectivity. Together with the previous character identity in Lemma \ref{lm_characterIdentity} we conclude bijectivity and hence the proof of Theorem~\ref{thm_VirasoroStructure}.

\bigskip 

We remark that on the associated graded module $L_n=L_n'+\xi\,\Shift_{n-1}$, and the associated graded Lie algebra is commutative, because all commutators have strictly lower degree. The $L_n'$ commute with each other and their action is well-studied in \cite{FF93}, in particular their dependencies when acting on a Fock module. The shift operators $\Shift_{n-1}$ commute with each other and their action can be described by power sums acting on alternating polynomials, in particular Newton's identity describes their dependencies. More sophisticated, we can use Schur polynomials to distinctively produce a particular basis element in $\SF^{\geq 0}$. In Section \ref{sec_constructive} we sketch from this and the explicit bijection in Lemma \ref{lm_characterIdentity} an alternative proof of Theorem \ref{thm_VirasoroStructure} that does not use irreducibility of the Whittaker module and is constructive in the sense that it produces inverse images.

 \begin{remark}\label{rem_triangular}
 The results in this article turn out to only depend on the formal type of the connection, and hence on the Birkhoff normal form.
 If one wishes to study in examples beyond Birkhoff normal form and with nontrivial Stokes matrices, a good starting point is to take all $A_k$ to be lower (or upper) triangular. This implies that there is still a horizontal filtration, instead of a grading. In particular the Whittaker vectors $w_k$ have contributions in lower vertical and  horizontal degree. 
A good explicit class of examples is 
$$\frac{\partial}{\partial z} 
+  \frac{1}{z^2}\begin{pmatrix} \xi & 0 \\ 0 & -\xi \end{pmatrix}
+  \frac{1}{z}\begin{pmatrix} \varepsilon & \tau \\ 0  & -\varepsilon \end{pmatrix},
\qquad \tau\neq 1.
$$ 
This connection  is discussed thoroughly in \cite{JLP76a}, note they use the coordinate $1/z$ around an irregular singularity at infinity. In particular their Corollary 5.1 states that for $2\varepsilon\not\in\N$ this is not equivalent to Birkhoff normal form and Lemma 3.4 gives the Stokes matrices in terms of Gamma-functions. If on the other hand $2\varepsilon\in-\N$, then one can find a finite gauge transformation to Birkhoff normal form. Note that the gauge transformation related to spectral flow in Example \ref{exm_gaugeSpectralflow}, which is analytic up to a simple pole in $z=0$, transforms this exceptional connections to the other exceptional connections in \cite{JLP76a} Corollary 5.1. 

In the appendix, we discuss the differential equation related to this example.  
 \end{remark}

\section{Definitions}\label{sec_definition}

\subsection{Local definition}

Consider the 2-dimensional vector space $\C^2$ with basis $e^a$ for $a\in\{+,-\}$ and standard symplectic form
$$(e^a,e^b)=\begin{pmatrix} 0 & 1 \\ -1 & 0\end{pmatrix}$$ 
preserved by the standard action of $\mathrm{Sp}_2=\mathrm{SL}_2$.
All matrices below have rows and arrows indexed by $a\in\{+,-\}$. 

\begin{definition}[Deformed mode algebra]\label{def_deformedModeAlgebra}
For a family of symmetric $2\times2$-matrices $C_k^{ab}$ with rows and columns indexed by $a,b\in\{+,-\}$ and $k\in \Z$ we define the Clifford algebra $\SF_{C}$ by generators $\psi^a_n$ with $a\in\{+,-\}$ and $n\in \Z$ and relations  
$$\{\psi^a_m,\psi^b_n\}=m(e^a,e^b)\delta_{m+n=0}+\sum_{k\in\Z} C_k^{ab}\delta_{m+n=k}$$
\end{definition}
For an $\sl_2$-connection on the formal punctured disc
$$\d+\sum_{n\in\Z} A_{-n-1}z^n,\qquad A_k\in\sl_2$$
we define $\SF_{\d+A}$ to be $\SF_C$ with the Clifford parameter 
$C^{ab}=(A_ke^a,e^b)$, or explicitly
\begin{align}\label{formula_convention}
C_k
&=\begin{pmatrix}
0& -1 \\ 1 & 0
\end{pmatrix}A_k
\quad\text{i.e.}\quad
\begin{pmatrix} C_k^{++} & C_k^{+-} \\C_k^{-+} & C_k^{--}\end{pmatrix}
=\begin{pmatrix} -A_k^{-+} & A_k^{++} \\ -A_k^{--} & A_k^{+-}\end{pmatrix}
=\begin{pmatrix} -A_k^{-+} & -A_k^{--} \\ A_k^{++} & A_k^{+-}\end{pmatrix}
\end{align}

\begin{example}\label{ex_PoissonOrderOne}
In the case of Poisson order $\leq 1$ and diagonal irregular term  we introduce the following notation for the irregular and regular singular terms 
$$C_k=0,k>1,\qquad 
C_1=\begin{pmatrix} 0 & \xi \\ \xi & 0\end{pmatrix},\qquad
C_0=\begin{pmatrix} \tau^+ & \varepsilon \\ \varepsilon & \tau^-\end{pmatrix}
$$
which correspond to the following connection 
$$\d
+\begin{pmatrix} \xi & 0 \\ 0 & -\xi\end{pmatrix}/z^2
+\begin{pmatrix} \varepsilon & \tau^- \\ -\tau^+ & -\varepsilon\end{pmatrix}/z
+\text{ regular terms}
$$
Regular connections correspond to $\xi,\varepsilon,\tau= 0$, regular singular connections correspond to $\xi=0$ (and $\tau$ or $\varepsilon$ nonzero) and irregular connections of Poisson order $1$ correspond to $\xi\neq 0$. In this case, we frequently consider the Clifford subalgebra $\SF_{\d+A}^{01}$ generated by $\psi^\pm_0,\psi^\pm_0$, which has then relations
\begin{align*}
\{\psi^\pm_1,\psi^\mp_0\} &= \xi\\
\{\psi^\pm_0,\psi^\mp_0\} &= \varepsilon\\
\frac12(\psi^\pm_0)^2=\{\psi^\pm_0,\psi^\pm_0\} &= \tau^\pm\\
\{\psi^\pm_1,\psi^\pm_1\}=\{\psi^\pm_1,\psi^\mp_1\} &= 0
\end{align*}
\end{example}

In the present article, we take Definition \ref{def_deformedModeAlgebra} as a definition, but we now want to briefly address, in which sense this should be thought of as a $(\d+A)$-twisted module: The essential axiom of the (local) definition of a vertex algebra modules is the \emph{Jacobi identity} or \emph{Borcherds identity}, respectively. A special case, which is in many cases equivalent, is the \emph{commutator formula}, see e.g. \cite{FBZ04} Section 5.1.5. For $g$-twisted modules, there are respective twisted versions of these identities. In out setting, we are in particular interested in versions that are formulated for actions of Lie algebras and which do not require the actions to be semisimple, and these are found in \cite{Huang10} and \cite{Bak16}, to which the reader is referred for details.

\begin{lemma}\label{lm_twistedCommutator}
The proposed (anti-)commutator formula \eqref{formula_twistedJacobi} for $(\d+A)$-twisted modules for the elements $\psi_{-1}^a1,\psi_{-1}^b1$ matches Definition \ref{def_deformedModeAlgebra} for the $(\d+A)$-twisted mode algebra.
\end{lemma}
\begin{proof}
In the right-hand side of formula \eqref{formula_twistedJacobi}, the only contributions for the vertex algebra $\SF$ are for $j=1$, since the action of polynomials in $A_k$ produce linear combinations of $\psi_{-1}^\pm1$ and for these elements the only nontrivial contributions in the $(j)$-product with $j\geq 0$ is $\psi^c_j\psi_{-1}^b1=j(e^c,e^b)\delta_{j-1=0}1$. The unit acts by $1_{m+n-j-k}^{\d+A}=\delta_{m+n-j-k,-1}$  Using this, we find as asserted 
\begin{align*}
\sum_{j\geq 0}\sum_{k\in\Z} \big(({m+A_k\choose j} a)_{(j)} b\big)_{(m+n-j-k)}^{\d+A}
&=m(e^a,e^b)\delta_{m+n,0}
+\sum_{k\in\Z}(A_ke^a,e^b)\delta_{m+n,k}\\
&=\{\psi_m^a,\psi_n^b\} \qedhere
\end{align*}
\end{proof}

\subsection{Global definition}\label{sec_defGlobal}

Let $\Sigma$ be a Riemann surface and $E\to \Sigma$ a symplectic vector bundle of rank $2$. For example, $E$ can be the trivial bundle $\Sigma\times \C^2$ with standard symplectic form on $\C^2$ with basis $e^+,e^-$ and the standard action of $\mathrm{Sp}_2=\SL_2$.
Let $\connection$ be an $\SL_2$-connection on $\Sigma$.

\begin{definition}
Define the Clifford algebra $\SF_{\connection}(E)$ as a central extension of the algebra of sections $f:\Sigma\to E$ by the anticommutator  
$$\{f,g\}=\Res( \connection f, g)=-\Res( f,\connection g)$$
\end{definition}

We want to show how this locally reproduces Definition \ref{def_deformedModeAlgebra}: Let $\Sigma=\C^\times$ and $E=\Sigma\times {\C^2}$ the trivial symplectic bundle, where ${\C^2}$ with standard basis $e_1,e_2$. Introduce basis elements $\psi_n^\pm=e^\pm z^{n}$ in the space of sections. Consider the connection $\connection=\d+A$ with 
$A=\sum_{k\in\Z} A_kz^{-1-n}$.
Then the Clifford algebra in this basis reads
\begin{align*}
\{\psi^a_m,\psi^b_n\}
&=\Res(e^a\, m z^{m-1}+\sum_{k\in\Z}A_ke^a\,z^{m-1-k}, \,e^b \, z^{n})\\
&=(e^a,e^b)m\delta_{m+n=0}+\sum_{k\in\Z}(A_ke^a,e^b)\delta_{m+n=k}
\end{align*}


\subsection{Gauge transformations}

Let $F\in \SL_2((z))$. As $\SL_2=\mathrm{Sp}_2$ acts on $\C^2$, preserving the standard symplectic form, the group $\SL_2((z))$ acts on series resp. functions $\C^\times\to {\C^2}$. As expected we find

\begin{lemma}\label{lm_gaugeModules}
The action of $F^{-1}$ gives an a (formal) isomorphism of Clifford algebras
$$\SF_{\d+A}\to \SF_{F[\d+A]}$$
where the connections transforms in the usual way, see e.g. \cite{Boalch02}
$$F[\d+A]=\d-F^{-1}(\d F)+FAF^{-1}$$
Accordingly, precomposition with $F$ gives a functor
$$\cC_{\d+A}\to \cC_{F[\d+A]}$$
\end{lemma}
\begin{proof}
This is a standard computation using that $F$ preserves the symplectic form
\begin{align*}
\{F^{-1}f,F^{-1}g\}&=\Res((\d+A) F^{-1}f,F^{-1}g) \\
&=\Res(F(\d+A) F^{-1} f,g) \\
&=\Res(F(\d+A) F^{-1} f,g) \\
&=\Res(F[(\d+A)] f,g) \qedhere
\end{align*}
\end{proof}

Note that we later consider modules annihilated by sufficitly high powers of $\psi^+$, then the action of $F\in \SL_2((z))$ produces only finitely many nonzero term when acting on an element of such a module.

\begin{example}\label{exm_gaugeBorel}
Consider the gauge transformation
$$F(z)=
\begin{pmatrix}
1 & z^N \\ 0 & 1   
\end{pmatrix}$$
Acting with $F^{-1}$ gives
\begin{align*}
\psi_n^+ &\longmapsto \psi_n^+ \\
\psi_n^- & \longmapsto \psi_n^- - \psi_{n+N}^+
\end{align*}
and, as an example, the following connection $\d+A$ is  transformed to $F[\d+A]$ as follows
\begin{align*}
&\d+\begin{pmatrix}
\xi & 0 \\ 0 & -\xi  
\end{pmatrix}/z^2
+\begin{pmatrix}
\varepsilon & 0 \\ 0 & -\varepsilon  
\end{pmatrix}/z\\
\longmapsto \;\;
&\d
+
\begin{pmatrix}
\xi & -2\xi z^N \\ 0 & -\xi   
\end{pmatrix}/z^2
+
\begin{pmatrix}
\varepsilon & -2\varepsilon z^N \\ 0 & -\varepsilon   
\end{pmatrix}/z
-
\begin{pmatrix}
0 & Nz^{N-1} \\ 0 & 0   
\end{pmatrix}
\end{align*} 
\end{example}

It is also instructive in what follows to allow more general gauge transformation with singularities in $z=0$:

\begin{example}\label{exm_gaugeSpectralflow}
Consider the singular gauge transformation
$$F(z)=
\begin{pmatrix}
z & 0 \\ 0 & z^{-1}   
\end{pmatrix}$$
Acting with $F^{-1}$ maps
\begin{align*}
\psi_n^+ &\longmapsto \psi_{n-1}^+\\
\psi_n^- & \longmapsto \psi_{n+1}^- 
\end{align*}
and, as an example, the following connection $\d+A$ is  transformed to $F[\d+A]$ as follows 
\begin{align*}
&\d+\begin{pmatrix}
\xi & 0 \\ 0 & -\xi  
\end{pmatrix}/z^2
+\begin{pmatrix}
\varepsilon & 0 \\ 0 & -\varepsilon  
\end{pmatrix}/z\\
\longmapsto \;
&\d+
\begin{pmatrix}
\xi & 0 \\ 0 & -\xi   
\end{pmatrix}/z^2
+
\begin{pmatrix}
\varepsilon & 0 \\ 0 & -\varepsilon
\end{pmatrix}/z
-
\begin{pmatrix}
1& 0 \\ 0 & -1   
\end{pmatrix}/z^{-1}
\end{align*}
Hence this causes a shift of $\varepsilon$ by $-1$. In particular, it changes $A_0$ but not the exponential $e^{2\pi\i A_0}$.  We would interpret this as  \emph{spectral flow}.
\end{example}

\section{Structure of twisted representations}\label{sec_twistedRepresentations}

The representation theory of the Clifford algebra $\SF_{C}$ in Definition \ref{def_deformedModeAlgebra} can in principle be read off the defining quadratic form, however, we are in an infinitely-dimensional case.  We now take a direct approach using a fixed triangular decomposition. 

\subsection{Irregular of Poisson order 1}

 If $\d+A$ has an irregular singularity of Poisson order $1$ with diagonalizable irregular term, as in Example \ref{ex_PoissonOrderOne}, then the elements $\psi^\pm_n$ for $n\geq 2$ generate a commutative subalgebra $\SF_{\d+A}^{\geq 2}$, and the pairing of this subalgebra with $\SF_{\d+A}^{\leq -1}$ is nondegenerate. Note that $\SF_{\d+A}^{\geq 1}$ is a larger commutative subalgebra. Of course, analogous statements holds for irregular singularities of higher order. 
Let $V$ be a representation of $\SF_{\d+A}$ and let $v$ be a vector annihilated by $\SF_{\d+A}^{\geq 1}$, i.e. $\psi_n^\pm v=0$ for $n>1$. As usual, there are universal representations generated by such vectors:

\begin{definition}
    Let $V^{01}$ be a representation of the algebra $\SF_{C}^{01}$ generated by $\psi_0^\pm,\psi_1^\pm$ and whose relations are given in Example \ref{ex_PoissonOrderOne}. Then $V^{01}$ extends trivially to a representation of $\SF_{\d+A}^{\geq 0}$, because $\SF_{\d+A}^{\geq 2}$ commutes with $\SF_{\d+A}^{01}$. Inducing up produces a representation of $\SF_{\d+A}$
    $$\Ind(V^{01})=\SF_{\d+A}\otimes_{\SF_{\d+A}^{\geq 0}} V^{01}$$
\end{definition}

The following behavior is typical for Clifford algebras with triangular decomposition (say, the Heisenberg Lie algebra) and considerably easier than the behavior of, say, finite semisimple Lie algebras:

\begin{lemma}
Let $v$ be a vector in $\Ind(V^{01})$ annihilated by $\SF_{\d+A}^{\geq 2}$, then $v\in V^{01}$. 
\end{lemma}
\begin{proof}
If we define the degree of $\psi_{-m}^\pm$ to be $m$, then $\Ind(V^{01})$ is non-negatively graded and the homogeneous component of degree $0$ is $V^{01}$.

Explicitly, $\psi_n^\pm$ for $n\geq 2$ acts on products of $\psi_m^\pm$ for $m\leq -1$ by an derivation $\smash{\xi\frac{\partial}{\partial \psi_{-n+1}}}$ of degree $-n+1$ plus operations of higher degree - for example in degree $-n$ the term $\smash{(\pm m)\frac{\partial}{\partial \psi_{-n}^\mp}}$ from the familiar untwisted case. Hence, if $v$ is annihilated by $\smash{\SF_{\d+A}^{\geq 1}}$, then the homogeneous component of largest degree has to be annihilated by  $\smash{\SF_{\d+A}^{\geq 1}}$ and already annihilated by the anti-derivations above. Now, the only polynomial annihilated by all derivations is the zero polynomial. Hence $v$ is in degree $0$ as asserted. 
\end{proof}

Since any morphism from an induced module is uniquely determined by the image of $V^{01}$, which has to be annihilated by $\SF_{\d+A}^{\geq 2}$, we conclude that all morphisms between induced modules 
$$f:\; \Ind(V^{01})\longrightarrow \Ind(W^{01})$$
are induced from a respective morphism $f^{01}:V^{01}\to W^{01}$. Spoken categorically

\begin{corollary}\label{cor_fullyfaithful}
The induction functor 
$$\Ind:\;\Rep(\SF_{\d+A}^{01}) \longrightarrow \Rep(\SF_{\d+A})$$
is fully faithful. 
Hence the full subcategory $\cC$ of representations in the image of $\Ind$ is equivalent to the category of representations of $\SF_{\d+A}^{01}$. In particular, the simple objects in $\cC$ are precisely the induced module $\Ind(V^{01})$ for $V^{01}$ a simple $\SF_{\d+A}^{01}$-module. 
\end{corollary}
It should not be hard to provide an intrinsic characterization of this category and clarify extensions, analogously to \cite{OW13}.

We hence discuss the representation category of $\SF_{\d+A}^{01}$: It is according to Example \ref{ex_PoissonOrderOne} the finite-dimensional Clifford algebra associated to the matrix

\[
  \left(\begin{array}{@{}cc|cc@{}}
    0 & 0 &       0 & \xi \\
    0 & 0 &       \xi & 0 \\ \hline
    0 & \xi &   \tau^+ & \varepsilon \\
    \xi & 0 &   \varepsilon & \tau^-
  \end{array}\right)
\]

This matrix is nondegenerate under our assumption $\xi\neq 0$, and a maximal isotropic subspace is spanned by $\psi_1^\pm$, regardless of the values $\tau^\pm, \varepsilon$.  

\begin{corollary}\label{cor_VNullEins}
Assume that $\d+A$ has an irregular singularity of order $1$ with diagonalizable irregular term, as in Example \ref{ex_PoissonOrderOne}. Then $\SF_{\d+A}^{01}$ is a simple algebra of dimension $2^4$. The unique simple module of dimension $2^2$ is $V^{01}=\SF_{\d+A}^{0}$, with $\psi_0^\pm$ acting by left-multiplication and $\psi_1^\pm$ acting by \smash{$\xi \frac{\partial}{\partial \psi_0^\mp}$}. Correspondingly, there is a unique simple module $\Ind(V^{01})$ of $\SF_{\d+A}$. As a vector space and $\SF_{\d+A}^{\leq 0}$-module we have
$$\Ind(V^{01})=\SF_{\d+A}^{\leq 0}$$
\end{corollary}
and $\psi^\pm_{n}$ for $n\geq 1$ act as anti-derivations, see formula \eqref{formula_positiveModesActing}.

\subsection{Regular singular}

As a comparison, for regular singular connections ($\xi=0$) we have a triangular decomposition with $\SF_{\d+A}^{\geq 1},\,\SF_{\d+A}^{0},\, \SF_{\d+A}^{\leq -1}$. The pairing is again nondegenerate,  regardless of the values of $\tau^\pm,\varepsilon$, and the respective version of Corollary \ref{cor_fullyfaithful} holds so the category $\cC$ is equivalent to the category of representations of $\SF_{\d+A}^{0}$. This is the finite-dimensional Clifford algebra associated to the matrix
\[
  \left(\begin{array}{@{}cc@{}}

   \tau^+ & \varepsilon \\
   \varepsilon & \tau^-
  \end{array}\right)
\]
The representation theory depends on this matrix being zero, being a Jordan block with eigenvalue zero, or being nondegenerate. This corresponds to the three types of twisted modules for symplectic fermions and the three types of fibres for the Kac-De-Concini quantum group discussed in \cite{FL24} Figure 1.

\subsection{Gauge transformations}

It is an important observation, that the gauge transformations in $\SL_2((z))$ by definition preserve $\SF_{\d+A}^{\geq k}$. In particular, the central subalgebra $\SF_{\d+A}^{\geq 2}$ is preserved, and the subalgebra $\SF_{\d+A}^{01}$ is preserved up to terms in $\SF_{\d+A}^{\geq 2}$. Hence, gauge transformations map the subcategory $\cC$ of induced modules to the respective subcategory $\cC$ for the transformed connection.

On the other hand, the singular gauge transformations we considered in general do not preserve the subcategory. For example, the spectral flow in Example \ref{exm_gaugeSpectralflow} maps the central subalgebra generated by $\psi_n^\pm,n\geq 2$ to the alternative subalgebra generated by $\psi_n^+,n\geq 1$ and $\psi_n^-, n\geq 3$.

\section{Virasoro action}\label{sec_Virasoro}

\subsection{Local definition}

The Sugawara-type construction has some subtleties, which are in principle present in the untwisted case (and lead to the central extension of the Virasoro algebra) but become more prominent in the twisted case: 

\begin{remark}
In principle we would like to define $L_n=\sum_{i+j=n} \psi^-_i\psi^+_j$, then the calculation below seems to prove $[L_m,L_n]=(m-n)L_{m+n}$. However, this expression causes infinities, for example on a highest weight vector $\psi^-_i\psi^+_{-i} v=(-i)v$ if $i>0$, already in the untwisted case. We will make a choice of normal ordering to resolve this problem, thereby effectively changing the summands of the naive choice $L_n$ by constants (that would sum to infinity). This causes in turn a central extension to appear. 
\end{remark}
\begin{definition}\label{def_normord}
    We define a \emph{normal ordering} as follows
   \begin{align*}
   \normord{\psi^a_i\psi_j^{b}}
   &=\begin{cases}
   \psi^a_i\psi_j^{b},\quad &i<j \\
   -\psi^{b}_j\psi_i^a,\quad &i>j \\
   \frac12\psi^a_i\psi_j^{b}-\frac12\psi^{b}_j\psi_i^a,\quad &i=j
   \end{cases}
   \end{align*}
   which satisfies $\normord{\psi^{b}_j\psi_i^{a}}=-\normord{\psi^a_i\psi_j^{b}}$.
\end{definition}
We introduce the notation
$$
\indic_{i\gtrsim j}:=\begin{cases}
   0,\quad &i<j \\
   \frac12,\quad &i=j \\
   1,\quad &i>j \\
   \end{cases}
   $$
\begin{align}\label{formula_normord}
\normord{\psi^a_i\psi_j^{b}}
&=\psi_i^a\psi_j^{b} \indic_{j\gtrsim i}
-\psi_j^{b} \psi_i^a \indic_{i\gtrsim j}\\
\psi^a_i\psi_j^{b}
   &=\normord{\psi^a_i\psi_j^{b}}
   +\{\psi^a_i,\psi_j^{b}\}\indic_{i\gtrsim j}
\end{align}

\begin{lemma}\label{lm_Vir}
The following operators in $\SF_{\d+A}$
$$L_{n}^{\d+A}:=\sum_{i+j=n}\normord{\psi_i^-\psi_j^+}
\,+\,c_n1,\qquad c_n=\frac12\sum_{i+j=n} \left(C_{i}^{-+}C_{j}^{+-}+C_{i}^{++}C_{j}^{--}\right)$$
fulfill the relations of the Virasoro algebra at central charge $c=-2$. 
\end{lemma}
\begin{proof}
We make the preliminary definition
$$L_{n}^{\d+A,\,pre}:=\sum_{i+j=n}\normord{\psi_i^-\psi_j^+}$$
Our main goal is to compute the commutator 
$$[L_{m}^{\d+A,\,pre},L_{n}^{\d+A,\,pre}]
=\sum_{\substack{i+j=m\\k+l=n}} [\psi_i^-\psi_j^+,\psi_k^-\psi_l^+]$$
We use summand-wise the following formula for the commutator 
\begin{align*}
[\psi_i^-\psi_j^+,\psi_k^-\psi_l^+]
&=\psi_i^-\{\psi_j^+,\psi_k^-\}\psi_l^+
-\{\psi_i^-,\psi_k^-\}\psi_j^+\psi_l^+
+\psi_k^-\psi_i^- \{\psi_j^+,\psi_l^+\}
-\psi_k^-\{\psi_i^-,\psi_l^+\}\psi_j^+
\end{align*}
We can express this commutator in terms of normally ordered products using formula  \eqref{formula_normord}
\begin{align*}
[\psi_i^-\psi_j^+,\psi_k^-\psi_l^+]
&=\{\psi_j^+,\psi_k^-\}\normord{\psi_i^-\psi_l^+}
-\{\psi_i^-,\psi_l^+\}\normord{\psi_k^-\psi_j^+}
+\{\psi_j^+,\psi_k^-\}\{\psi_i^-,\psi_l^+\}
(\indic_{i\gtrsim l}-\indic_{k\gtrsim j})\\
&+\{\psi_j^+,\psi_l^+\}\normord{\psi_k^-\psi_i^-}
-\{\psi_i^-,\psi_k^-\}\normord{\psi_j^+\psi_l^+}
-\{\psi_i^-,\psi_k^-\}\{\psi_j^+,\psi_l^+\}(\indic_{k\gtrsim i}-\indic_{j\gtrsim l})
\end{align*}
We plug into this formula the explicit form of the anticommutators 
$$\{\psi^a_i,\psi_j^{b}\}=(a i) \delta_{i+j=0}\delta_{-a=b}+C_{i+j}^{a b}$$ 
We now discuss all terms appearing in this way.
Note that for fixed $m,n$ only finitely many terms have $\indic_{i\gtrsim l}-\indic_{k\gtrsim j}\neq 0$.
\begin{enumerate}[a)]
\item $j\normord{\psi_i^-\psi_l^+}$ and $-l\normord{\psi_k^-\psi_j^+}$ \newline
In the second term we relabel $i,j,k,l$ with $(m-n)+k,l,i,(n-m)+j$, which preserves the sum conditions $i+j=m$ and $k+l=n$. Then  we are left with $\big(j-((n-m)+j)\big)\normord{\psi_i^-\psi_l^+}$ and thus with the overall sum 
$$(m-n)L_{m+n}^{\d+A,\,pre}.$$
\item $C_{j+k}^{+-}\normord{\psi_i^-\psi_l^+}$ and $-C_{i+l}^{-+}\normord{\psi_k^-\psi_j^+}$\newline
In the second term we again relabel $i,j,k,l$ with $(m-n)+k,l,i,(n-m)+j$, which preserves the sum conditions and maps the index $i+l$ to $j+k$. Then by $C^{+-}=C^{-+}$ the two summands cancel.
\item $C_{j+l}^{++}\normord{\psi_k^-\psi_i^-}$ and similarly 
$-C_{i+k}^{--}\normord{\psi_j^+\psi_l^+}$\newline
In the first term, we relabel $i,j,k,l$ with $k,(m-n)+l,i,(n-m)+j$, which also preserves the sum conditions and preserves the index $j+l$. Then by the antisymmetry $\normord{\psi_k^-\psi_i^-}=-\!\normord{\psi_i^-\psi_k^-}$ the sum over the first term vanishes individually. Similarly, the sum over the second terms vanishes individually.
\item $j\delta_{j+k=0}(-i)\delta_{i+l=0}(\indic_{i\gtrsim l}-\indic_{k\gtrsim j})$\newline
The sum conditions $i+j=m$ and $k+l=n$ together with $j+k=0$ and $i+l=0$ leave for the tuples $i,j,k,l$ only $i,m-i,-m+i,-i$ together with the condition $m+n=0$. The term $(\indic_{i\gtrsim l}-\indic_{k\gtrsim j})$ is $+1$ if $m>i>0$ and $+1/2$ if one and only one of the inequalities hold as equalities, but in these boundary cases the summand $j(-i)$ vanishes. Similarly the term is $-1$ resp. $-1/2$ if $m<i<0$. For $m>0$ only the case $+1$ appears and the sum over $i$ is explicitly
\begin{align*}
 \sum_{m>i>0} (m-i)(-i)
&= \frac{(m-1)m(2m-1)}{6}-\frac{(m-1)m}{2}m
=-\frac{m^3-m}{6}
\end{align*}
For $m=0$ the condition is empty and for $m<0$ we get the negative result, hence this formula also holds in these cases. We are left with the overall contribution $-\frac{m^3-m}{6}\delta_{m+n=0}$, corresponding to the Virasoro algebra anomaly at central charge $c=-2$.
\item $j\delta_{j+k=0}C_{i+l}^{-+}(\indic_{i\gtrsim l}-\indic_{k\gtrsim j})$ and $C_{j+k}^{+-}(-i)\delta_{i+l=0}(\indic_{i\gtrsim l}-\indic_{k\gtrsim j})$\newline
The sum conditions $i+j=m$ and $k+l=n$ together with $j+k=0$ leave for the tuples $i,j,k,l$ only $m-j,j,-j,n+j$, hence  $C_{i+l}^{-+}=C_{m+n}^{-+}$ and the term $(\indic_{i\gtrsim l}-\indic_{k\gtrsim j})$ is $+1$ if $0<j<(m+n)/2$. Similarly, for the second term the condition $i+l=0$ leaves only  $i,m-i,n+i,-i$, hence  $C_{i+l}^{-+}=C_{m+n}^{-+}$ and the term $(\indic_{i\gtrsim l}-\indic_{k\gtrsim j})$ is $+1$ if $0<i<(m+n)/2$. Apparently, these terms again cancel.  

\item $C_{j+k}^{+-}C_{i+l}^{-+}(\indic_{i\gtrsim l}-\indic_{k\gtrsim j})
$ and  $-C_{i+k}^{--}C_{j+l}^{++}(\indic_{k\gtrsim i}-\indic_{j\gtrsim l})$\newline
Relabeling $i,j$ in the second term gives 
$$\left(C_{j+k}^{+-}C_{i+l}^{-+}+C_{j+k}^{--}C_{i+l}^{++}\right)
(\indic_{i\gtrsim l}-\indic_{k\gtrsim j})$$
We now count the number of contributions for fixed indices of $C$: We claim
$$\sum_{\substack{i,j,k,l \\ i+j=m,\;k+l=n\\i+l=a,\;j+k=b}} (\indic_{i\gtrsim l}-\indic_{k\gtrsim j})
=\frac{m-n}{2}\delta_{a+b=m+n}$$
We now show this claim: The conditions leaves for the tuple $i,j,k,l$ only $i,m-i,n-a+i,a-i$ and the condition $a+b=m+n$. The term $(\indic_{i\gtrsim l}-\indic_{k\gtrsim j})$ is $+1$ for $(m-n)/2+a/2>i>a/2$ and $+1/2$ if one and only one of these inequalities hold as equalities, and respectively for $-1,-1/2$. If we assume $m>n$ then we have four cases: If $m-n$ is even and $a$ is odd, then the inequality has $(m-n)/2$ solutions $i$ with $+1$. If $m-n$ is even and $a$ is even, then the inequality has $(m-n)/2-1$ solutions $i$ with $+1$ and two boundary solutions $i$ with $+1/2$. If $m-n$ is odd, then the inequality has $(m-n-1)/2$ solutions $i$ with $+1$ and one boundary solutions $i$ with $+1/2$, depending on $a$ being even or odd. In all cases the sum is $(m-n)/2$, which proves the claim. The same result also arises for $m<n$ and $m=n$. As a consequence, we have the overall contribution
$$\frac{m-n}{2}\sum_{\substack{a,b\\a+b=m+n}} \left(C_{a}^{-+}C_{b}^{+-}+C_{a}^{++}C_{b}^{--}\right)$$
\end{enumerate}
Collecting the nonvanishing terms a), d), f) we have the overall result
\begin{align*}
[L_{m}^{\d+A,\,pre},L_{n}^{\d+A,\,pre}]
&=(m-n)L_{m+n}^{\d+A,\,pre}
-\frac{m^3-m}{6}\delta_{m+n=0}\\
&+\frac{m-n}{2}\sum_{\substack{a,b\\a+b=m+n}} \left(C_{a}^{-+}C_{b}^{+-}+C_{a}^{++}C_{b}^{--}\right)
\end{align*}
The nonvanishing terms a) and d) are the usual Virasoro relations at central charge $c=-2$. The additionally nonvanishing term f) is of the form $(m-n)c_{m+n}1$, so it is a Lie algebra coboundary, which can be easily removed by redefining 
$$L_n^{\d+A}:=L_n^{\d+A,\,pre}+c_{n}1$$
Note that adding a constants to either argument of the commutator does not change the result.
\end{proof}

\subsection{Global definition}\label{sec_VirasoroGlobal}

As a consequence of the definition, we find

\begin{lemma}\label{lm_trafoGlobal}
$$\left[L_n^{\d+A},\begin{pmatrix}\psi_k^+ \\\psi_k^-\end{pmatrix}\right]
=(-k)\begin{pmatrix}\psi_{n+k}^+ \\\psi_{n+k}^-\end{pmatrix}
+\sum_{i\in\Z}
\begin{pmatrix} -C^{-+}_{n-i+k} & C^{++}_{n-i+k} \\ 
-C^{--}_{n-i+k}  & C^{+-}_{n-i+k} \end{pmatrix}
\begin{pmatrix}\psi_i^+ \\\psi_i^-\end{pmatrix}
$$
\end{lemma}
\begin{proof}
\begin{align*}
[\psi_i^a\psi_j^b,\psi_k^c]
&=\psi_i^a \{\psi_j^b,\psi_k^c\}
-\{\psi_i^a,\psi_k^c\} \psi_j^b
\\
&=\psi_i^a\left( (bj)\delta_{j+k=0}\delta_{-b=c}+C_{j+k}^{bc} \right)
- \psi_j^b\left( (ai)\delta_{i+k=0}\delta_{-a=c}+C_{i+k}^{ac} \right)
\end{align*}
For the normal ordered product we get essentially the same result:
\begin{align*}
[\normord{\psi_i^-\psi_j^+},\psi_k^c]
&=
\psi_i^-\big( (+j)\delta_{j+k=0}\delta_{-=c}+C_{j+k}^{+c} \big)\indic_{j\gtrsim i}
- \psi_j^+\big( (-i)\delta_{i+k=0}\delta_{+=c}+C_{i+k}^{-c} \big)\indic_{j\gtrsim i}\\
&\;-\psi_j^+\big( (-i)\delta_{i+k=0}\delta_{+=c}+C_{i+k}^{-c} \big)\indic_{i\gtrsim j}
+ \psi_i^-\big( (+j)\delta_{j+k=0}\delta_{-=c}+C_{j+k}^{+c} \big)\indic_{i\gtrsim j}\\
&=\psi_i^-\big( j\delta_{j+k=0}\delta_{-=c}+C_{j+k}^{+c} \big)
+\psi_j^+\big(i\delta_{i+k=0}\delta_{+=c}+C_{i+k}^{-c} \big)
\end{align*}
We sum over $i+j=n$ and separate the terms. In the second term we switch the role of $i,j$ and in the third term we apply the delta-expressions.
$$
\psi_i^-C_{n-i+k}^{+c}
-\psi_i^+C_{n-i+k}^{-c}
-k\psi_{n+k}^c   
$$
The constant $c_n\cdot 1$ has no effect in the commutator. Hence this proves the first assertion on $L_n^{\d+A}$. 
\end{proof}

This result can be read as a transformation behavior of a field. In particular $L_{-1}$ acts as the connection: 

\begin{corollary}
If we consider $\psi(z)$ as a function with in the dual $\sl_2$-representation $\overline{\C^2}$, that is, acting by $-A^T$, then the previous result reads\footnote{We seem to not be able to use the same conventions for the Definition of the Clifford algebra in Section \ref{sec_defGlobal} and the Virasoro action in this Section \ref{sec_defGlobal}, probably because one are functions and one are functionals.}

$$[L_{-1}, \psi(z)]=\left(\frac{\partial}{\partial z}+A(z)\right){\psi}(z)$$

\end{corollary}
For the proof we check that with the identification of $C$ and $A$ in formula \eqref{formula_convention}
$$\begin{pmatrix} -C^{-+}_{j} & C^{++}_{j} \\ 
-C^{--}_{j}  & C^{+-}_{j} \end{pmatrix}
=\begin{pmatrix} A^{--}_j & -A^{-+}_j \\ 
-A^{+-}_j  & A^{++}_j \end{pmatrix}
=-A^T
$$
then $A(z)\psi(z)=\sum_{j,i}A_j\psi_i z^{-1-j-1-i}$ and this recovers the right-hand side of the previous lemma for $j=n-i+k$ for $n=-1$ as coefficient of $z^{-1-k}$.

\subsection{Formulas in the general case}\label{sec_Formulas}

We make the formulas for $L_n^{\d+A}$ more explicit in the case of an irregular singularity of Poisson order $1$ as in Example \ref{ex_PoissonOrderOne} and the induced $\SF_{\d+A}$-representation $\Ind(V^{01})$ in Corollary \ref{cor_VNullEins}, which as a vector space coincides with $\SF_{\d+A}^{\leq 0}$. 

The generators $\psi_n^\pm,\,n\leq0$ act by left-multiplication, note that the algebra relations among themselves depend on the terms $A_k,\,k\leq 0$ in $\d+A$, and only the relations among the $\psi_0^\pm$ depends on the regular singular term $A_0$. The generators $\psi_m^\pm,\,m>0$ act as an anti-derivation $\{\psi_m^\pm,-\}$, and with the explicit terms in Example \ref{ex_PoissonOrderOne} we can write this action 
\begin{align}
  \label{formula_positiveModesActing}
  \begin{split}
\psi_{m>0}^\pm 
\longmapsto 
\xi \frac{\partial}{\partial \psi_{-(m-1)}^\mp}
+(\pm m+\varepsilon)\frac{\partial}{\partial \psi_{-m}^\mp}
&+\sum_{k=-1}^{-\infty} C_k^{\pm\mp}\frac{\partial}{\partial \psi_{-(m-k)}^\mp} \\
+\tau^\pm\frac{\partial}{\partial \psi_{-m}^+}
&+\sum_{k=-1}^{-\infty} C_k^{\pm\pm}\frac{\partial}{\partial \psi_{-(m-k)}^+}
\end{split}
\end{align}
To get an explicit expression for $L_n^{\d+A}$, we plug this into Lemma \ref{lm_Vir} and formula \eqref{formula_normord}
$$L_n^{\d+A}
=
c_n\cdot 1+
\sum_{i+j=n} \normord{\psi_i^-\psi^+_j}
= 
c_n\cdot 1
+
\sum_{i+j=n} 
\psi_i^-\psi^+_j\,\indic_{j \gtrsim i}
-\psi_j^+\psi^i_i\,\indic_{i \gtrsim j}
$$
To get an overview over the resulting terms, we introduce the following symbolic notation for different types of summands
\begin{align*}
\left(L_n^{\d+A}\right)_{\psi,\psi}
&:=\sum_{\substack{i+j=n \\ i,j\leq 0}} \normord{\psi_i^-\psi^+_j}\\
\left(L_n^{\d+A}\right)_{\partial,\partial}
&:=\sum_{\substack{i+j=n \\ i,j> 0}} \normord{\psi_i^-\psi^+_j}\\
\left(L_n^{\d+A}\right)_{\psi,\partial}
&:=\sum_{i\leq 0,j> 0} \psi_i^-\psi^+_j
-\sum_{i> 0,j\leq 0} \psi_j^+\psi_i^-
\end{align*}
Note that $\left(L_n^{\d+A}\right)_{\psi,\psi}=\left(L_n\right)_{\psi,\psi}$ agrees with the untwisted term and only appears for $n\leq 0$. Note that in $\left(L_n^{\d+A}\right)_{\partial,\partial}=\left(L_n\right)_{\psi,\psi}$ the normally ordered product coincides with the usual product, and the term appears only for $n\geq 2$. With the explicit expression for $\psi_{m>0}^\pm$ in formula \eqref{formula_normord} we get many terms besides the untwisted term (involving $m$), namely: shifting the degree by $-1$ (involving $\xi$), by $0$ (involving $\varepsilon, \tau^\pm$) and by $-k>0$  (involving $C_k^{\pm\pm}$). We indicate these terms by denoting the shift in degree on the symbol $\partial$. Altogether this gives the following terms for $L_n^{\d+A}$, which we sort by degree ($0$, $n-2$, $n-1$, $n$,...)

\begin{align*}
L_n^{\d+A}
&=c_n\cdot 1 \\
&+\left(L_n^{\d+A}\right)_{\partial-1,\partial-1} \\
&+\left(L_n^{\d+A}\right)_{\psi,\partial-1} 
+\left(L_n^{\d+A}\right)_{\partial-1,\partial}
+\left(L_n^{\d+A}\right)_{\partial-1,\partial+0}\\
&+\underbrace{\left(L_n\right)_{\psi,\psi} +\left(L_n\right)_{\psi,\partial}+ \left(L_n\right)_{\partial,\partial}}_{L_n}
+\left(L_n^{\d+A}\right)_{\partial-1,\partial+1} 
+\left(L_n^{\d+A}\right)_{\psi,\partial+0}
+\left(L_n^{\d+A}\right)_{\partial+0,\partial+0} \\
&+\cdots
\end{align*}

We spell out some of these terms more explicitly
\begin{itemize}
\item We have from the formula in Lemma \ref{lm_Vir}:
$$c_2=\frac12 \xi^2,\qquad
c_1=\frac12 \xi \varepsilon,\qquad
c_0=\frac12 (\varepsilon^2+\tau^+\tau^-)+\frac12 \xi C_1^{+-},\qquad
\cdots$$
\item The unique term in degree $n-2$ reads:
$$\left(L_n^{\d+A}\right)_{\partial-1,\partial-1}=\xi^2 \sum_{\substack{i+j=n\\i,j\geq 1}}\frac{\partial}{\partial \psi^+_{-(i-1)}}\frac{\partial}{\partial \psi^-_{-(j-1)}}$$
\item In degree $n-1$ we have the following terms
\begin{align*}
\left(L_n^{\d+A}\right)_{\partial-1,\partial}
&=\sum_{\substack{i+j=n\\i\geq 1,j\geq 0}}\xi\frac{\partial}{\partial \psi^+_{-(i-1)}}
j\frac{\partial}{\partial \psi^-_{-j}}
+(-i)\frac{\partial}{\partial \psi^+_{-i}}
\xi\frac{\partial}{\partial \psi^-_{-(j-1)}}\\
&=\xi \sum_{\substack{i+j=n-1\\i\geq 0,j\geq 0}} (j-i) \frac{\partial}{\partial \psi^+_{-i}}\frac{\partial}{\partial \psi^-_{-j}} \\
\left(L_n^{\d+A}\right)_{\partial-1,\partial+0}
&=\xi \sum_{\substack{i+j=n-1\\i\geq 0,j\geq 0}}2\varepsilon\frac{\partial}{\partial \psi^+_{-i}}
\frac{\partial}{\partial \psi^-_{-j}}
+\tau^+\frac{\partial}{\partial \psi^+_{-i}}
\frac{\partial}{\partial \psi^+_{-j}}
+\tau^-\frac{\partial}{\partial \psi^-_{-i}}
\frac{\partial}{\partial \psi^-_{-j}}
\end{align*}

\item The terms $\left(L_n^{\d+A}\right)_{\psi,\partial-k}$ in degree $n-k$ for $k\leq 1$  read 

$$\left(L_n^{\d+A}\right)_{\psi,\partial-k}=C^{+-}_k \sum_{\substack{i+(j-k)=n-k \\ i\leq 0,\;j-k\geq 0}} \psi_i^-\frac{\partial}{\partial \psi^-_{-(j-k)}}-\psi_i^+\frac{\partial}{\partial \psi^+_{-(j-k)}}
=C^{+-}_k\,\Shift_{n-k}$$
where the constant is $C_k^{+-}=\xi$ for $k=1$ and $C_k^{+-}=\varepsilon$ for $k=0$. We have introduced an derivation given on generators by
$$\Shift_{n-k}:\quad \psi^\pm_j\mapsto \mp\psi^\pm_{j+(n-k)}$$
Note that in the usual Virasoro action the term $\left(L_n\right)_{\psi,\partial}$ is a similar shift by $n$, but there involves an additional factor $\mp j$.
\item Further terms in degree $>n$ arise only in presence of regular terms $C_k^{\pm\pm},k<0$.
\end{itemize}
Altogether we can list these first terms for some small important cases
\begin{align*}
L_2^{\d+A} 
&=\frac12 \xi^2 \cdot 1+ \xi^2\frac{\partial}{\partial \psi^+_{0}}\frac{\partial}{\partial \psi^-_{0}} \\
&+\xi\Shift_1 + \xi \sum_{\substack{i+j=n-1\\i\geq 0,j\geq 0}} (j-i) \frac{\partial}{\partial \psi^+_{-i}}\frac{\partial}{\partial \psi^-_{-j}}  \\
&+L_2+\cdots \\
&\\
L_1^{\d+A} 
&=\frac12 \xi \varepsilon \cdot 1+ \xi\Shift_0 \\
&+L_1+\cdots \\
&\\
L_0^{\d+A} 
&= \xi \Shift_{-1} \\
&+ (\frac12 (\varepsilon^2-\tau^+\tau^-)+\frac12 \xi C^{+-}_1)\cdot 1 + L_0 \\
&+\cdots
\end{align*}

In particular for vectors of top degree, the action of $L_2,L_1,L_0$ is completely determined by the shown terms. 

\subsection{Formulas in the formal case}

We now assume the connection is of the simplest diagonal form with no regular terms i.e. in  Birkhoff normal form
$$\d
+\begin{pmatrix} \xi & 0 \\0 & -\xi\end{pmatrix}/z^2
+\begin{pmatrix} \varepsilon & 0 \\0 & -\varepsilon\end{pmatrix}/z
$$
Formally, every irregular connection of Poisson order $1$ with diagonalizable irregular term is equivalent to a connection of this form. Analytically, this is only true if also the Stokes matrices are trivial. The differential equation has a basis of solutions
$$\begin{pmatrix} e^{\lambda/z}z^{-\varepsilon} \\ 0 \end{pmatrix},\quad
\begin{pmatrix} 0 \\ e^{-\lambda/z}z^\varepsilon  \end{pmatrix}$$

The formulas in the previous section simplify now as follows: The action of the positive modes in formula \eqref{formula_positiveModesActing} becomes

\begin{align}
  \label{formula_positiveModesActingDiagonal}
\psi_{m>0}^\pm 
\longmapsto 
\xi \frac{\partial}{\partial \psi_{-(m-1)}^\mp}
&+(\pm m+\varepsilon)\frac{\partial}{\partial \psi_{-m}^\mp}
\end{align}

The formula for $L_n^{\d+A}$ again follows from Lemma \ref{lm_Vir} and formula \eqref{formula_normord}, where the constant shift is $c_2=\frac{1}{2}\xi^2$ and $c_1=\frac{1}{2}\xi\varepsilon$ and $c_0=\frac{1}{2}\varepsilon^2$  and $c_k=0$ for $k\leq 0$. Altogether we get terms in degree $0$, $n-2$, $n-1$, $n$ as follows:

\begin{align}\label{formula_VirasoroDiagonal}
\begin{split}
L_n^{\d+A} 
&=c_n\cdot 1 \\
&+\xi^2\underbrace{\sum_{\substack{i+j=n\\i\geq 1,j\geq 1}}\frac{\partial}{\partial \psi^+_{-(i-1)}}\frac{\partial}{\partial \psi^-_{-(j-1)}}}_{\left(L_n^{\d+A}\right)_{\partial-1,\partial-1}} \\
&+\underbrace{\xi \sum_{\substack{i+j=n \\ i\leq 0,\;j\geq 1}} \psi_i^-\frac{\partial}{\partial \psi^-_{-(j-1)}}-\psi_i^+\frac{\partial}{\partial \psi^+_{-(j-1)}}}_{\left(L_n^{\d+A}\right)_{\psi,\partial-1}\,=\,\xi\,\Shift_{n-1}}
+\underbrace{\xi \sum_{\substack{i+j=n-1 \\ i\geq 0,\;j\geq0}} (j-i) \frac{\partial}{\partial \psi^+_{-i}}\frac{\partial}{\partial \psi^-_{-j}}}_{\left(L_n^{\d+A}\right)_{\partial-1,\partial}}
+\underbrace{2\xi\varepsilon \sum_{\substack{i+j=n-1 \\ i\geq 0,\;j\geq0}}\frac{\partial}{\partial \psi^+_{-i}}
\frac{\partial}{\partial \psi^-_{-j}}}_{\left(L_n^{\d+A}\right)_{\partial-1,\partial+0}} \\
&+L_n
+\underbrace{\varepsilon \sum_{\substack{i+j=n \\ i\leq 0,\;j\geq0}} \psi_i^-\frac{\partial}{\partial \psi^-_{-j}}-\psi_i^+\frac{\partial}{\partial \psi^+_{-j}}}_{\left(L_n^{\d+A}\right)_{\psi,\partial+0}\,=\,\varepsilon\,\Shift_{n}}
+\underbrace{\varepsilon^2 \sum_{\substack{i+j=n \\ i\geq 0,\;j\geq0}}\frac{\partial}{\partial \psi^+_{-i}}
\frac{\partial}{\partial \psi^-_{-j}}}_{\left(L_n^{\d+A}\right)_{\partial+0,\partial+0}}
\end{split}
\end{align}
where we again use the derivation $\Shift_{n-k}:\; \psi^\pm_j\mapsto \mp\psi^\pm_{j+(n-k)}$.

The terms without a distinguished name only appear for $n\geq 0$. We list the first cases:

\begin{lemma}\label{lm_VirasoroDiagonal}
For a connection in Birkhoff normal form we have the following explicit formulas. The linebreaks collect terms of degree $n-2$, $n-1$ and $n$: 
\begin{align*}
L_2^{\d+A} 
&=\frac12 \xi^2 \cdot 1+ \xi^2\frac{\partial}{\partial \psi^+_{0}}\frac{\partial}{\partial \psi^-_{0}} \\
&+\xi\,\Shift_1 
+
(-\xi+2\xi\varepsilon)
\frac{\partial}{\partial \psi^+_{-1}}\frac{\partial}{\partial \psi^-_{0}} 
+
(\xi+2\xi\varepsilon)
\frac{\partial}{\partial \psi^+_{0}}\frac{\partial}{\partial \psi^-_{-1}} \\
&+L_2
+\varepsilon\,\Shift_2
+\varepsilon^2 \sum_{\substack{i+j=2 \\ i\geq 0,\;j\geq0}}\frac{\partial}{\partial \psi^+_{-i}}
\frac{\partial}{\partial \psi^-_{-j}}\\
L_1^{\d+A} 
&=\frac12 \xi \varepsilon \cdot 1
+ \xi\,\Shift_0 
+ 2\xi\varepsilon\,
\frac{\partial}{\partial \psi^-_{0}}\frac{\partial}{\partial \psi^-_{0}} \\
&+L_1
+ \varepsilon\,\Shift_1
+\varepsilon^2 \sum_{\substack{i+j=2 \\ 1\geq 0,\;j\geq0}}\frac{\partial}{\partial \psi^+_{-i}}
\frac{\partial}{\partial \psi^-_{-j}}\\
L_0^{\d+A} 
&= \xi \,\Shift_{-1} \\
&+ L_0
+ \frac12 \varepsilon^2\cdot 1 
+\varepsilon\,\Shift_0
+\varepsilon^2 \frac{\partial}{\partial \psi^+_{0}}
\frac{\partial}{\partial \psi^-_{0}}
\\
L_n^{\d+A} 
&= \xi \,\Shift_{n-1} \\
&+L_n+\varepsilon\,\Shift_0, \qquad n<0\\
\end{align*}
\end{lemma}

\subsection{Examples}\label{sec_Examples}

We now assume that $\d+A$ is in Birkhoff normal form and assume in addition  $\varepsilon=0$. Using the formulae \eqref{formula_VirasoroDiagonal} and the simplified formulae in  Lemma \ref{lm_VirasoroDiagonal} for $n\leq 2$ we compute and depict the first terms of the action on the (used-to-be-)familiar top vectors of the spaces of horizontal degree $k=0$ and $k=1$. Recall that  in Birkhoff normal form the horizontal grading is preserved.
\begin{itemize}
\item (Figure \ref{figure_ExampleVir1}). For the top vector $v$ all derivatives and shifts in formula  \eqref{formula_VirasoroDiagonal} are zero and only the constant term $c_0$ and the untwisted terms $L_n$ remain. The constant term turns $v$ into a a (slightly degenerate) Whittaker vector, a topic we study further in Section \ref{sec_VirasoroStructure}.
\begin{align*}
    L_n^{\d+A} v  &=0,\;n>2\\ 
    L_2^{\d+A} v &=\frac{1}{2}\xi^2 v\\ 
    L_1^{\d+A} v &=0\\ 
    L_0^{\d+A} v &=L_0 v =\psi^-_0\psi^+_0v \\
    L_{-1}^{\d+A} v &= L_{-1} v=\psi^-_{-1}\psi^+_0v +\psi^-_0\psi^+_{-1}v
\end{align*}
 Note that the induced module for $A=0$ is not the vacuum module, but an indecomposable projective module of symplectic fermions, see Example \ref{exm_regular}. So at this point we expect an irreducible module containing $v$ extended by modules containing $\psi_0^+,\psi_0^-$ and $\psi_0^-\psi^+$. In particular already in the untwisted case $L_0$ acts by a Jordan block and also  $L_{-1}$ act nonzero.
 
Now if we apply $L_0^{\d+A}$ again to $L_0^{\d+A} v$ then $L_0$ acts by zero, but we have one new term:
$$(L_0^{\d+A})^2 v=\xi\,\Shift_{-1}(\psi^-_0\psi^+_0)v
=\xi\left(\psi^-_{-1}\psi^+_0-\psi^-_0\psi^+_{-1}\right)v$$
This matches the expectation from Whittaker modules (that all $(L_0)^nv$ are linearly independent) and also provides some indication that the twisted action turns the indecomposable projective module in the untwisted case  (Example \ref{exm_regular}) becomes simple under the twisted action, as the $\psi^\pm_0v$ do not define submodules, not even for the Virasoro action.

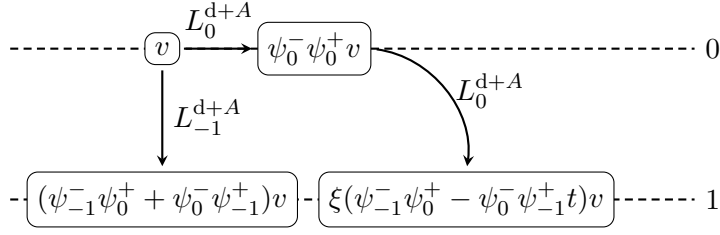
\begin{figure}\label{figure_ExampleVir1}
\centerline{
    \begin{tikzpicture}[
      scale=0.5,
      degree/.style={thick,densely dashed},
      trans/.style={thick,->,shorten >=2pt,shorten <=2pt,>=stealth},
      roundnode/.style={fill=white,draw,rounded corners}
    ]
    \begin{scope}[xshift=0cm]
    \newcommand{\from}{-4cm}
    \newcommand{\till}{14cm}
    \newcommand{\high}{4cm}
    \draw[degree] (\from,-0*\high) -- (\till,-0*\high) node[right] {$0$};
    \draw[degree] (\from,-1*\high) -- (\till,-1*\high) node[right] {$1$};
    \node[roundnode] (v) at (0,0) {$v$};   
    \node[roundnode] (logv) at (4,0) {$\psi_0^-\psi_0^+v$};  
    \node[roundnode] (Tv) at (0,-1*\high) {$(\psi^-_{-1}\psi^+_0+\psi^-_0\psi^+_{-1})v$};
    \node[roundnode] (LLv) at (8,-1*\high){$\xi(\psi^-_{-1}\psi^+_0-\psi^-_0\psi^+_{-1}t)v$};
    \draw[trans] (v)--(logv) node[midway,above] {$L_0^{\d+A}$};
    \draw[trans] (v)--(Tv) node[midway,right] {$L_{-1}^{\d+A}$};
    \draw[trans] (logv) edge[bend left=45] node[midway,right] {$L_{0}^{\d+A}$} (LLv) ;

    \end{scope}
    
    \end{tikzpicture}
}
\caption{Some decendents of $v$ in the $(\d+A)$-twisted module. The curved arrows are the new terms.}
\end{figure}

\item (Figure \ref{figure_ExampleVir2}) Next we start with $\psi_0^{-}$. All double derivatives and positive shifts are zero, but in addition to the constant term and the untwisted terms $L_n$ we have $\Shift_k(\psi_0^{-})=\psi_{-k}^{-}$. Again we find a Whittaker vector, with different parameters 
\begin{align*}
    L_n^{\d+A} \psi_0^{-} v &=0,\;n>2\\ 
    L_2^{\d+A} \psi_0^{-} v &=\frac{1}{2}\xi^2 \,\psi_0^{-}v\\ 
    L_1^{\d+A} \psi_0^{-} v &= \xi \,\psi_0^{-}v\\ 
    L_0^{\d+A} \psi_0^{-} v &=\xi\,\psi_{-1}^{-}v + L_0 \psi_{0}^{-}v 
    =\xi\psi_{-1}^{-}v  \\
    L_{-1}^{\d+A} \psi_0^{-} v &= \xi\,\psi_{-2}^{-}v + L_{-1} \psi_{0}^{-}v
    =\xi\,\psi_{-2}^{-}v - \psi_{-1}^{-}\,\psi_{0}^{-}\psi_{0}^{+}v\\
    &\\
    (L_0^{\d+A})^2 \psi_0^{-} v&=\xi^2\psi_{-2}^-v + \xi\,L_0\psi_{-1}^{-}v
    =\xi^2\psi_{-2}^-v + \xi(\psi_{-1}^{-}+
    \psi_{-1}^{-}\,\psi_0^-\psi_0^+)v
\end{align*}
Again, we find indications that the twisted module is simple. 
\end{itemize}

\begin{figure}\label{figure_ExampleVir2}
\centerline{
    \begin{tikzpicture}[
      scale=0.5,
      degree/.style={thick,densely dashed},
      trans/.style={thick,->,shorten >=2pt,shorten <=2pt,>=stealth},
      roundnode/.style={fill=white,draw,rounded corners}
    ]
    \begin{scope}[xshift=0cm]
    \newcommand{\from}{-9cm}
    \newcommand{\till}{6cm}
    \newcommand{\high}{4cm}
    \draw[degree] (\from,-0*\high) -- (\till,-0*\high) node[right] {$0$};
    \draw[degree] (\from,-1*\high) -- (\till,-1*\high) node[right] {$1$};
    \draw[degree] (\from,-2*\high) -- (\till,-2*\high) node[right] {$2$};
    \node[roundnode] (v) at (0,0) {$\psi_0^{-1}$};   
    \node[roundnode] (Lv) at (-6,-1*\high) {$\xi\psi_{-1}^{-} $};
    \node[roundnode] (Tv) at (0,-1*\high){$-\psi^-_{-1}\,\psi_0^-\psi^+_0$};
    \node[roundnode] (defTv) at (0,-2*\high){$\xi \psi^-_{-2}$};
    

    \draw[trans] (v) edge[bend right=35] node[midway,left] {$\;L_{0}^{\d+A}$} (Lv) ;
    \draw[trans] (v) -- (Tv) node[midway,right] (L) {$L_{-1}^{\d+A}$};
    \draw[trans] (L) edge[bend left=80] (defTv);
    \draw[trans] (Lv) -- (Tv) node[midway,below] (LL) {$L_0^{\d+A}$};
    \draw[trans] (LL) edge[bend right=40] (defTv);

    \end{scope}
    
    \end{tikzpicture}
}
\caption{Some decendents of $\psi_0^-$ in the $(\d+A)$-twisted module. The curved arrows are the new terms.}
\end{figure}
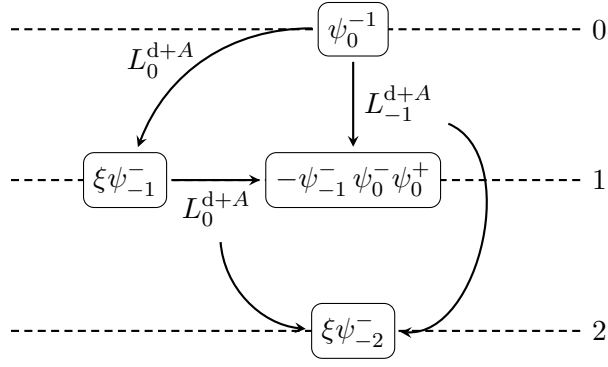

\section{Structure as a Virasoro module}\label{sec_VirasoroStructure}

\subsection{Preliminaries on Whittaker modules}

For Whittaker modules for the Virasoro algebra we refer the reader to
\cite{OW09}, for considerations regarding the category and extensions to \cite{OW13} and  in particular a treatment of higher order to \cite{LGZ11}.

We recall that, for the Virasoro algebra, for any choice of $1$-dimensional representation $\alpha:\Vir^{\geq 1}\to \C$ we define a \emph{Whittaker vector} $v_\alpha$ of type $\alpha$ by the condition 
$$L_nv_\alpha=a_n v_\alpha,\quad n\geq 1$$
for $a_n=\alpha(L_n)$. Differently said,  $v_\alpha$ spans the $1$-dimensional representation of $\Vir^{\geq 1}$. For the Virasoro algebra, the commutator relation forces $a_n=0$ for $n\geq 3$, and we denote shorthand $\alpha=(a_1,a_2)$.

Similar to Verma modules, there is a \emph{universal Whittaker module}
$$M_{(a_1,a_2)}=\Vir\otimes_{\Vir^{> 0}}\C_{(a_1,a_2)}
=\C[L_0,L_{-1},\ldots]v_{(a_1,a_2)}
$$

By \cite{LGZ11} Theorem 7 we have

\begin{theorem}\label{thm_WhittakerIrrep}
    The universal Whittaker module $M_{(a_1,a_2)}$ is irreducible for all values $(a_1,a_2)$ except $(0,0)$, which correspond to highest weight modules. 
\end{theorem} Note that the statements in \cites{OW09,OW13} do not cover the degenerate case $(a_1,a_2)=(0,\frac{1}{2}\xi^2)$ which we encounter in Corollary \ref{cor_WhittakerVectors} below for $k=0$.

\bigskip
 Note that \cite{OW13} show the possibility of extensions between Whittaker modules, if the Whittaker parameters are equal. It would be interesting to find situations where these extensions appear, for example maybe if $A_1$ is nilpotent.

\subsection{Whittaker vectors}\label{sec_Whittaker}

As a consequence of the general formulas for $L_n^{\d+A}$ for $n\geq 0$ in Section \ref{sec_Formulas}, we find as in the example in Section \ref{sec_Examples}.

\begin{corollary}\label{cor_WhittakerVectors}
    Assume that all $A_k$ are diagonal, for example the connection is in Birkhoff normal form, so the horizontal grading is preserved. Then the following vectors (which are the highest-weight vectors for the untwisted Virasoro action)
    \begin{align*}
        v_{k}&=\psi_{-k}^-\psi^-_{-(k-1)}\cdots \psi^-_0 \\
        v_0&=v\\
        v_{-k}&=\psi^+_{-k}\psi^+_{-(k-1)}\cdots \psi^+_0 
    \end{align*}    
    become for the twisted Virasoro action Whittaker vectors
    \begin{align*}
        L_{n\geq 3}^{\d+A} v_k &=0   \\ 
        L_2^{\d+A} v_k
        &=\frac12 \xi^2 \cdot v_k \\
        L_1^{\d+A} v_k 
        &=\xi(k+\frac12 \varepsilon)\cdot v_k
    \end{align*}
\end{corollary}

In the general case, acting with $L_{n}^{\d+A},\; n\geq 1$ on $v_k$ produces additional terms in lower vertical degree, possibly in different horizontal degree than $k$. Since the degree is bounded from below, there exist finite sums starting with $v_k$, which are simultaneous eigenvectors.

\begin{corollary}\label{cor_WhittakerVectorsII}
In general, there exists for every $k\in \Z$ a Whittaker vector $w_k$ equal to $v_k$ above plus terms in lower degree, such that:
    \begin{align*}
        L_{n\geq 3}^{\d+A} w_k &=0   \\ 
        L_2^{\d+A} w_k
        &=\frac12 \xi^2 \cdot w_k \\
        L_1^{\d+A} w_k 
        &=\xi(k+\frac12 \varepsilon)\cdot w_k
    \end{align*}
\end{corollary}

\subsection{A character identity}

Our expectation that the twisted action turns the induced module (for $A=0$ the indecomposable projective module) into a Whittaker module suggests a nontrivial character identity, as discussed in the introduction. This identity can be proven directly and can be used to prove the isomorphism in the next section:

\begin{lemma}\label{lm_characterIdentity}
We have the following identities of $q$-series for all $k$ and all $M,N$
\begin{align}\label{formula_qVand}
&q^{k(k-1)/2} {M+N \choose M-k}_q\\
&=\sum_{x\geq 0} q^{(k+x)(k+x-1)/2} {M \choose k+x}_q
\cdot q^{x(x-1)/2} {N \choose x}_q \cdot q^x
\end{align}
as well as 
\begin{align}
&q^{k(k-1)/2} {M+N-1 \choose M-k}_q
+q^{(k+1)k/2} {M+N-1 \choose M-k-1}_q\\
&=\sum_{x\geq 0} q^{(k+x)(k+x-1)/2} {M \choose k+x}_q
\cdot q^{x(x-1)/2} {N \choose x}_q 
\end{align}
\end{lemma}

\begin{proof}
All claims follow essentially from  the $q$-Vandermond identity, see \cite{And98} Theorem~3.4, and reflection:
$${M+N \choose K}_{q} =\sum_{J} {M \choose K-J}_q {N \choose J}_q q^{J(M-K+J)}
 =\sum_{J} {M \choose M-K+J}_q {N \choose J}_q q^{J(M-K+J)}
$$
For the first assertion, we directly take $J=x$ and $K=M-k$ and convince ourselves that the $q$-powers are as intended $q^{J(M-K+J)}=q^{x(k+x)}$ and another term independent of~$x$:
\begin{align*}
q^{(k+x)(k+x-1)/2}q^{x(x-1)/2}q^x
&=q^{x(k+x)}q^{k(k-1)/2}
\end{align*}
For the second assertion, we first use the $q$-Pascal identity, see \cite{And98} Theorem 3.2, so the right-hand side reads
\begin{align*}
&\sum_{x\geq 0} q^{(k+x)(k+x-1)/2} {M \choose k+x}_q
\cdot q^{x(x-1)/2} q^x {N-1 \choose x}_q\\
+& \sum_{x\geq 0} q^{(k+x)(k+x-1)/2} {M \choose k+x}_q
\cdot q^{x(x-1)/2} {N-1 \choose x-1}_q
\end{align*}
Then on both terms we can apply the $q$-Vandermonde identity, 
with $N-1$ replacing $N$, and once with $J=x$ and $K=M-k$ and once with $J=x-1$ and $K=M-k-1$, and again we convince ourselves that the $q$-powers are as intended in both cases with $q^{J(M-K+J)}=q^{x(k+x)}$ respectively $q^{J(M-K+J)}=q^{(x-1)(k+x)}$:
\begin{align*}
q^{(k+x)(k+x-1)/2}q^{x(x-1)/2}q^x
&=q^{x(k+x)}q^{k(k-1)/2} \\
q^{(k+x)(k+x-1)/2}q^{x(x-1)/2}
&=q^{(x-1)(k+x)}q^{(k+1)k/2} 
\end{align*} 
\end{proof}

We can interpret ${M+N \choose M}_q=\frac{[M+N]_q!}{[M]_q![N]_q!}$ as the $q$-series where the coefficient of $q^n$ is the cardinality of the set $P(n,M,N)$ of partitions of $n$ into at most $M$ summands with each summand at most $N$, see \cite{And98} Theorem 3.1. We write such a partition $X_1^{m_1}\cdots X_{N}^{m_N}$ with $\sum m_i\leq M$. We now give in this context an interpretation of the $q$-Vandermonde identity \eqref{formula_qVand}, which is probably known:

\begin{lemma}\label{lm_constructive}
Let $X=X_1^{m_1}\cdots X_{k+N}^{m_{k+N}}$ be a restricted partition in $P(n,M-k,N+k)$. Let $x\leq N$ be the largest integer, such that 
\begin{align}\label{formula_xdefinable}
\sum_{i\geq k+x} m_i \geq x
\end{align}
Then $X$ can be uniquely decomposed into a partition $X_1^{m_1'}\cdots X_{k+x}^{m_{k+x}'}$ in  $P(n',M-k-x,k+x)$ and a partition  \smash{$X_{k+x}^{m_{k+x}''}\cdots X_{k+N}^{m_{k+N}''}$} of $n''$ with precisely $x$ parts. The latter can be be reinterpreted, by shifting the index by $k+x$ (possibly to zero), as a partition in $P(n''-x(k+x),x,N-x)$. 
\end{lemma}
\begin{proof}
The definition of $x$ as the largest integer fulfilling formula \eqref{formula_xdefinable} is unique and it is well defined because for $x=0$ the formula certainly holds.

For the existence and uniqueness of the decomposition we note that the definition of $x$ implies that $\sum_{i\gneq k+x} m_i \leq x$, thus there is exists a unique subpartition with precisely $x$ parts $\geq k+x$ that contains all parts $\gneq k+x$, namely  \smash{$X_{k+x}^{m_{k+x}''}\cdots X_{k+x}^{m_{k+x}''}$} for $m_{i}''=m_{i}$ for $i\gneq k+x$ and $m_{k+x}''=x-\sum_{i\gneq k+x} m_i$.
\end{proof}

\subsection{Main theorem}

As discussed in the introduction,  we take now on the induced module $\SF_{\d+A}^{\leq 0}$ the grading  
$$\deg_{total}=\deg_{vert}+\deg_{fine}$$
and on the  $\Vir^{\leq 0}$ \marginpar{on Virasoro algebra or Whittaker modules?} we take the grading $\deg_{Vir}(L_n^{\d+A})=-n+1$, then the explicit formulas in Section \ref{sec_Formulas} shows that the action is compatible with these filtration, and on the associated graded module we have $L_n=L_n'+\xi\,\Shift_{n-1}$, where the shift operator  raises  $\deg_{vert}$ by $-n+1$  and preserves $\deg_{fine}$, and 
$$L_n'=\left(L_n^{\d+A}\right)_{\psi,\psi}
=\sum_{\substack{i+j=n \\ i,j\leq 0}} {\psi_i^-\psi^+_j}$$
(with normal order irrelevant) is independent of the twisting and it raises  $\deg_{vert}$ by $-n$ any $\deg_{fine}$ by $1$, as it creates an additional pair $\psi^+\psi^-$. For any Whittaker vector $w_k$ we get a $\Vir$-homomorphism
from the universal Whittaker module
\begin{align}\label{formula_fk}
\begin{split}
f_k:\, M_{(a_1,a_2)}=\C[L_0,L_{-1},\ldots]
&\longrightarrow \SF_{\d+A}^{\leq 0}\\
\quad\; L_{n_1}\cdots L_{n_t}v
&\longmapsto L_{n_1}^{\d+A}\cdots L_{n_t}^{\d+A}w_k
\end{split}
\end{align}
The reader may again imagine that $\d+A$ is in Birkhoff normal form, then the horizontal grading is preserved and the Whittaker vectors $w_k$ are the explicit top vectors $v_k$ in Corollary \ref{cor_WhittakerVectors}.

In Theorem \ref{thm_WhittakerIrrep} we recalled that the universal Whittaker modules $M_{(a_1,a_2)}$ for parameters $(a_1,a_2)\neq (0,0)$ are irreducible. In particular, the homomorphisms $f_k$ in formula \eqref{formula_fk} are injective, and because their top terms $w_k$ are linearly independent, the sum $\sum_k f_k$ is injective as well. The map preserves the filtration by $\deg_{total}$, and our character identity in Lemma \ref{lm_characterIdentity}  shows that $\sum_k f_k$ is bijective.

\begin{theorem}\label{thm_VirasoroStructure}
For $\d+A$ an $\sl_2$-connection, irregular of Poisson order $1$ with diagonalizable irregular term as in Example~\ref{ex_PoissonOrderOne}, the simple induced module $\SF_{\d+A}^{\leq 0}$ is, as a Virasoro representation, the direct sum of Whittaker modules with generator $w_k,k\in\Z$
  \begin{align*}
        L_{n\geq 3}^{\d+A} w_k &=0   \\ 
        L_2^{\d+A} w_k
        &=\frac12 \xi^2 \cdot w_k \\
        L_1^{\d+A} w_k 
        &=\xi(k+\frac12 \varepsilon)\cdot w_k
    \end{align*}
\end{theorem}

\subsection{A more constructive proof}\label{sec_constructive}

We would also like to sketch a more constructive and explicit version of  Theorem \ref{thm_VirasoroStructure}, in the sense that it provides a filtration and an explicit bijection of leading terms, and thus an inductive definition of the inverse of the bijection.

In Lemma \ref{lm_constructive} we have given an alternative explicit proof for the character identity by providing a bijection between $\C[X_1,\ldots,X_{k+N}]$ to the tensor product of $\C[X_1,\ldots,X_{k+x}]$ and a span of certain restricted partition $P(n,N-x,x)$. The goal is turn this to a description of leading terms in $\gr(f)$ in formula \eqref{formula_fk}, where 
$$X_{n+1}=L_n^{\d+A}=\Shift_{-(n+1)}+L_n'$$
$$L_n'=(L^{\d+A}_n)_{\psi\psi}=\sum_{i+j=n} \psi^-_i\psi^+_j$$
We consider in the image a monomial
$$\psi_{n_t}^-\cdots \psi_{n_1}^- \, \psi_{m_s}^+\cdots \psi_{m_1}^+$$
and assume without restriction of generality $s\leq t$ and define $k=t-s$. 

As a first step, we settle the case $N=0$, that is, we prove that all monomials with $s=0$ can be reached from $v_k$ by the shift operators: Consider the ring of polynomials in $x_1,\ldots, x_t$. 

\begin{lemma}
    Identify $\Shift_n$ with the symmetric polynomial given by the power sum $p_n=x_t^n+\cdots + x_1^n$ and identify $\psi_{m_t+t-1}^++\psi_{m_{t-1}+t-2}\cdots \psi_{m_1}^+$ for $m_1\geq m_2\geq \cdots \geq m_t$ with the alternating  polynomial
$$\det\begin{pmatrix} x_1^{m_t+t-1} & x_2^{m_t+t-1} & \cdots \\
x_1^{m_{t-1}+t-2} & x_2^{m_{t-1}+t-2} &  \\
\vdots & & \ddots
\end{pmatrix}$$
Then the action of the shift operator corresponds to the action of a symmetric polynomial on an alternating polynomial by multiplication.
\end{lemma} 
\begin{proof}
This would follow from the more general Murnaghan–Nakayama rule, but let us give an explicit argument: Write the determinant as a sum over permutations of $\prod_i x_\sigma(i)^{m_i+i-1}$ and multiply by $\sum_j x_j^n$. We may relabel depending on $\sigma$ our index $j=\sigma(j')$, then all summands with a fixed $j'$ collect to the determinant with $m_{j'}$ replaced by $m_{j'}+n$. Possibly, we rearrange the order of the index set, which leads to the expected additional sign. 
\end{proof}

It is known that the module of alternating polynomials is free of rank $1$, generated by the Vandermonde determinant, which corresponds to $v_k$ under the identification above. Since the power sums are a set of generators of the space of symmetric functions, their linear combinations can reach any monomial $\psi_{n_t}^-\cdots \psi_{n_1}^-v$ form $v$. 
\begin{remark}
We can be more specific: The previous determinant is equal to the \emph{Schur polynomial} $s_\lambda$ for the partition $\lambda=(m_t,\ldots,m_1)$ times the Vandermonde determinant, and the Schur polynomials are a linear basis of the space of symmetric functions. In particular, they can be expressed as polynomials in the power sums, and in fact the coefficient is the character of Specht module $S_\lambda$ of $\mathbb{S}_n$ evaluated at a conjugacy class given by a cycle type. Hence, if we regard the Schur polynomial $s_\lambda$ as a polynomial in the power sums, and in turn as a polynomials in the shift operators, then 
    $$s_\lambda.v=\psi_{\lambda_t+t-1}^-\cdots \psi_{\lambda_1}^-v$$
\end{remark}

On the other hand, the algebraic dependencies of shift operators (that is, here, kernel of the map $f$ restricted to the $s=0$ piece) are given by dependencies of the power sums $p_n$, which in turn follow from the alternative set of generators by elementary symmetric polynomials $e_i$, with transformation given by the  Newton identities.  The only algebraic dependencies of the $e_n$ are $e_n=0$ for $n>k$. Of course again here, using Schur polynomials would be more sophisticated and achieve the same purpose.

\begin{example}
If $k=1$ then $x_1^mx_1^{m'}=x_1^{m+m'}$ and accordingly $\Shift_m\Shift_{m'}=\Shift_{m+m'}$, and if $k=2$ then 
$$0=3e_3
=\frac12(x_1+x_2)^3
-\frac32(x_1^2+x_2^2)(x_1+x_2)
+(x_1^3+x_2^3)$$
and accordingly the shifts have such dependencies.
\end{example}

As a next step, we go from $N=0$ to $N=1$ and hence study contributions of the form  
$$\psi_{n_t}^-\cdots \psi_{n_1}^- \, \psi_{0}^+$$
and ignore all further terms. All our shift operators used in the previous step potentially produce now also terms of this form, where one of the shift operators $\Shift_{-n-1}$ is replaced by $L_n'$, note that shift operators and $L_n'$ do not commute. On the other hand, we now have $L_n^{\d+A}=\Shift_{-n-1}+L'_n$ additionally for $n=k$. In principle $\Shift_{-k-1}$ is algebraically dependent on $\Shift_{-1},\ldots,\Shift_{-k}$ and $L_k'v$ can be reached by shifts of $L_n'v$ for $k<n$. The point is that nevertheless the sum is (supposedly) a new generator, algebraically independent of the generators for $k<n$, due to subtle factors and signs appearing.

\begin{example}
We have 
\begin{align*}
    L_k^{\d+A}v&=\Shift_{-k-1}v+ \sum_{i+j=-k}\psi_{-i}^-\psi_{-j}^+\\
    &=(-1)^k \psi_{-k}^-\cdots \psi_{0}^- \, \psi_{0}^+
\end{align*}
We compare this to a suitable polynomial in  $L_n^{\d+A},n<k$ acting on $v$, which we construct to have the same first term: This is given by the Newton identity on $k$ variables 
$$p_{k+1}v
=\sum_{i=1}^k (-1)^{k+i}
p_ie_{k+1-i}
$$
Since $\psi_{-k}^-\cdots \psi_{0}^- \, \psi_{0}^+$ is the lowest term of its kind, it cannot be reached by shifts. Hence, if we interpret the right-hand side as a polynomial in $L_n^{\d+A},n<k$, the only terms that shift $\deg_{fine}$ by one are those where the leading $p_i$ is replaced by $L_{i-1}'$ and all $e_i$ by corresponding shift operators. We directly compute all contributions
\begin{align*}
   \left( \sum_{i=1}^k (-1)^{k+i} L_{i-1}'e_{k-(i-1)} \right)
    \psi_{-(k-1)}^-\cdots \psi_{0}^- 
    &=\sum_{i=1}^k (-1)^{k+i} L_{i-1}'
    \psi_{-k}^-\cdots \bcancel{\psi_{-(i-1)}^-}\cdots \psi_{0}^- \\
    &=\sum_{i=1}^k (-1)^{k+i} \psi_{-(i-1)}^-\psi_{0}^+\;
    \psi_{-k}^-\cdots \bcancel{\psi_{-(i-1)}^-}\cdots \psi_{0}^- \\
    &=\sum_{i=1}^k (-1)^{k+i} (-1)^{i-1}
    \psi_{-k}^-\cdots \psi_{0}^-\;\psi_0^+ \\
    &=(-1)^{k+1}k
    \psi_{-k}^-\cdots \psi_{0}^-\;\psi_0^+ 
\end{align*}
Now the prefactors appearing in $L_k^{\d+A}$ and the polynomial in  $L_n^{\d+A},n<k$ are different $(-1)^k\neq (-1)^{k+1}k$. Hence the two are linearly independent, and a suitable combination has only contributions in $\deg_{fine}\geq 1$. Note that a similar behavior can be forced more directly by using elementary symmetric polynomials $e_n,n\geq k$ or Schur polynomials $s_\lambda$, which is probably the more systematic approach, but we found the $L'_i$-contributions more difficult to determine. 

Similar calculations can be done for $s=1$ involving $\psi_{-i}^+$ for $i>0$, the typical effect is that shifts act with a sign on the $\psi^+$, while the direct terms $L_{k+i}$ are more symmetric in~$\psi^\pm$. 
\end{example}

\begin{problem}
It would be nice to fabricate a full alternative proof of Theorem \ref{thm_VirasoroStructure} from these ideas. Probably this needs more insight in Schur polynomials and how they change by replacing $p_n$ with $L_{n}'$ and/or a better use of the filtration given by $M$ in Lemmata \ref{lm_characterIdentity} and \ref{lm_constructive}.
\end{problem}

\section{Appendix: A differential equation}

We discuss a simplest example of a differential equation with an irregular singularity of Poisson order $1$ that is not in Birkhoff normal form and exhibits nontrivial Stokes matrices. In the context of our current article, it corresponds to the example discussed in Remark \ref{rem_triangular}.

Explicitly, this differential equation is solved by the incomplete Gamma function, and at integer parameters the exponential integral. At negative parameters, we encounter simple exponential functions, which matches that the Stokes matrix, which can be explicitly computed in terms of Gamma-functions, is zero in these cases.  

 The reader should compare the general treatment of irregular $\sl_2$-connections in \cite{JLP76a}, where this example appears as an exceptional case (note again the substitution $1/z$). The matching facts about the incomplete Gamma-functions we have found in \cite{Gau98}. Note that this example can be obtained by confluence from  the  hypergeometric functions, see e.g. \cite{Hor20}, which for us is relevant in view of Example \ref{exm_confluence} and Problem \ref{prob_confluence} and the work \cite{GT12}. 
\bigskip

We consider the following differential equation:
$$\left( \frac{\partial}{\partial z} 
+  \frac{1}{z^2}\begin{pmatrix} \xi & 0 \\ 0 & -\xi \end{pmatrix}
+ \frac{1}{z} \begin{pmatrix} \varepsilon & \tau \\ 0 &  -\varepsilon \end{pmatrix} \right)\Psi(z) = 0
$$ 
For $z$ sufficiently small, $A(z)$ has disctinct eigenvalues, here $\pm(\xi/z^2+\varepsilon/z)$, so it is similar to the respective diagonal matrix. We substitute correspondingly
$$\Psi(z)=\Phi(z)\exp\left(\frac{1}{z}\begin{pmatrix} \xi & 0 \\ 0 & -\xi \end{pmatrix}
- \log(z) \begin{pmatrix} \varepsilon & 0 \\ 0 &  -\varepsilon \end{pmatrix}\right)$$
and get the following differential equation 
$$\frac{\partial}{\partial z} \Phi(z)
-\Phi(z)\left(\frac{1}{z^2}\begin{pmatrix} \xi & 0 \\ 0 & -\xi \end{pmatrix}
+\frac{1}{z} \begin{pmatrix} \varepsilon & 0 \\ 0 &  -\varepsilon \end{pmatrix}\right)
+\left(\frac{1}{z^2}\begin{pmatrix} \xi & 0 \\ 0 & -\xi \end{pmatrix}
+ \frac{1}{z} \begin{pmatrix} \varepsilon & \tau \\ 0 &  -\varepsilon \end{pmatrix}\right)\Phi(z)=0
$$ 
Writing $\Phi(z)=\begin{pmatrix} A(z) & B(z) \\ C(z) & D(z)\end{pmatrix}$ we get 
$$\frac{\partial}{\partial z} \begin{pmatrix} A(z) & B(z) \\ C(z) & D(z)\end{pmatrix}
+\begin{pmatrix} 0 & 2(\xi/z^2+\varepsilon/z)B(z) \\ -2(\xi/z^2+\varepsilon/z)C(z) & 0 \end{pmatrix}
+\begin{pmatrix} 0 & \tau/z \\ 0 &  0 \end{pmatrix}\begin{pmatrix} A(z) & B(z) \\ C(z) & D(z)\end{pmatrix}=0
$$
\bigskip

{\bf Analytically}, this  system of differential equations with can be solved as follows:
\begin{align*}
\Phi(z)
&=
\begin{pmatrix} A(z) & B(z) \\ C(z) & D(z)\end{pmatrix}
=\begin{pmatrix}
1
& 
B(z) \\
0
& 
1
\end{pmatrix}\\
\Psi(z) &=
\begin{pmatrix}
1
& 
B(z) \\
0
& 
1
\end{pmatrix}
\begin{pmatrix} 
e^{\xi/z}z^{-\varepsilon} 
&
0
\\
0
& 
e^{-\xi/z}z^{\varepsilon} 
\end{pmatrix}
=
\begin{pmatrix} 
e^{\xi/z}z^{-\varepsilon} 
&
B(z)e^{-\xi/z}z^{\varepsilon} 
\\
0
& 
e^{-\xi/z}z^{\varepsilon} 
\end{pmatrix}
\end{align*}
with $B(z)$ given by the differential equation 
$$\frac{\partial}{\partial z} B(z)+2(\xi/z^2+\varepsilon/z)B(z)+(\tau/z)=0.$$
The homogeneous solution is $\exp(2\xi/z)z^{-2\varepsilon}$ and variation of constants and the substitution $t=2\xi/z$ leads quickly to the following slightly subtle integral:
\begin{align*}
B(z)&=-\tau\exp(2\xi/z)z^{-2\varepsilon}
\int_0^z \exp(-2\xi/z)z^{2\varepsilon-1} \d z\\
&=\tau\exp(2\xi/z)z^{-2\varepsilon}
\int_{2\xi/z}^\infty \exp(-t) t^{-2\varepsilon-1} (2\xi)^{2\varepsilon-1} \frac{\d t}{-t^2/2\xi}\\
&=-\tau\exp(2\xi/z) (z/2\xi)^{-2\varepsilon} \;
\Gamma(-2\varepsilon,2\xi/z).
\end{align*}
The integral involves the (upper) incomplete Gamma function $\Gamma(s,x)$. In particular, for an integer values of $s=-2\varepsilon$ we have a polynomial resp. an exponential integral 
$$\Gamma(s,x)=\begin{cases} \displaystyle
    (s-1)! e^{-x} \sum_{k=0}^{s-1} \frac{x^k}{k!},\qquad  &s=1,2,\ldots \\
    \displaystyle
    x^s\mathrm{E}_{1-s}(x),\qquad &1-s=1,2,\ldots \\
\end{cases}$$
We have chosen the integration boundary such that $\lim_{z\to 0}B(z)=0$, as an initial condition, but this only holds in a certain sector of the complex plane: It is known, as a consequence of Watson's lemma, that we have an asymptotic for $|\arg(x)|<\frac32 \pi$ and $|x|\to \infty$ as follows:
$$\Gamma(s,x)\sim x^{s-1} e^{-x}\sum_{k\geq 0} (s-1)\cdots(s-k)x^{-k}.$$
while the overlapping regions such as $|\arg(z)-2\pi m|<\frac32 \pi$ this asymptotic is shifted by the monodromy of the integrand
$$\Gamma(s,ze^ {2\pi\i\, m})
=e^ {2\pi\i\, ms}\Gamma(s,z)+(1-e^ {2\pi\i\, ms})\Gamma(s),\quad a\neq 0,-1,-2,\cdots
$$
Clearly, this asymptotic  expression shows that only for $\Re(x)>0$ resp. in the sector 
$|\arg(x)|<\frac12 \pi$ we have the aspired asymptotics 
$\lim_{z\to 0}B(z)=0$,
while in other sectors we have a constant shift, and therefore a different initial value. 
\begin{align*}
\Phi(z)
&=
\begin{pmatrix}
1
& 
B(z)\; -\;\tau(1-e^ {2\pi\i\, ms})\Gamma(s)\exp(2\xi/z) (2\xi/z)^{2\varepsilon} \\
0
& 
1 
\end{pmatrix}.
\end{align*}
This phenomenon of starting with sectorially different asymptotics, is called {Stokes phenomenon} and the  rays $\arg(x)=\pm \frac12 \pi$ are called {Stokes directions}.

\bigskip
{\bf As formal series}, we can alternatively make a power series ansatz 
$$\Phi(z)=\sum_{n\geq 0} \Phi_n z^n
=\sum_{n\geq 0} \begin{pmatrix} A_n & B_n \\ C_n & D_n\end{pmatrix} z^n
$$
giving us a recursion relation
$$n\Phi_n
+\left[ \begin{pmatrix} \xi & 0 \\ 0 & -\xi \end{pmatrix},\Phi_{n+1} \right] 
+\left[ \begin{pmatrix} \varepsilon & 0 \\ 0 & -\varepsilon \end{pmatrix},\Phi_n \right] 
+ \begin{pmatrix} 0 & \tau \\ 0 & 0 \end{pmatrix}\Phi_n =0
$$
$$n\begin{pmatrix} A_n & B_n \\ C_n & D_n\end{pmatrix}
+\begin{pmatrix} 0 & 2\xi B_{n+1} \\ -2\xi C_{n+1} & 0 \end{pmatrix}
+\begin{pmatrix} 0 & 2\varepsilon B_n \\ 2\varepsilon C_n & 0 \end{pmatrix} 
+ \begin{pmatrix}  \tau C_n & \tau D_n \\ 0 & 0 \end{pmatrix}=0
$$
One can verify elementary that this is solved by $\Phi_0=1$ and for $n>0$
\begin{align}\label{formula_Phin}
\Phi_n=
\begin{pmatrix}
0 & \tau \,(2\varepsilon+1)\cdots (2\varepsilon+n-1)
\\
0 & 0
\end{pmatrix}
(-2\xi)^{-n}
\end{align}
The problem is that this formal power series 
$$B(z)=\tau\sum_{n>0} (2\varepsilon+1)\cdots (2\varepsilon+n-1)\,(-2\xi)^{-n}$$
has convergence radius $0$, except for $2\varepsilon\in -\mathbb{N}$. In these degenerate cases one has indeed a polynomial substitution to the Birkhoff normal form, and hence the associated Clifford algebras $\SF_{\d+A}$ are isomorphic as in Lemma \ref{lm_gaugeModules}.  

{Borel resummation} can be used to produce from this divergent power series a convergent power series, but it requires the choice of a sector. 

The Stokes factors, which measures the transformation of the asymptotically well-behaved solutions between the sectors, can be formally defined, and depends on the growth rate of $\Phi_n$, see \cite{JLP76a} Theorem II
\begin{align*}
S_++S_-
&=
\lim_{n\to \infty}
\Phi_nK_n^{-1}\\
&=
\begin{pmatrix}
0 & \tau \,(2\varepsilon+1)_{n-1} (-2\xi)^{-n}
\\
0 & 0
\end{pmatrix}
\begin{pmatrix}
(-1)^n n^{-2\varepsilon} & 0 \\
0 & n^{2\varepsilon}
\end{pmatrix}^{-1}(-2\xi)^{n}/(n-1)!\\
&=
\lim_{n\to \infty} 
\begin{pmatrix}
0 & \tau \,(2\varepsilon+1)_{n-1} n^{-2\varepsilon}/(n-1)!
\\
0 & 0
\end{pmatrix}\\
&=
\begin{pmatrix}
0 & \tau\Gamma(2\varepsilon+1)^{-1} \\
0 & 0
\end{pmatrix}
\end{align*}
using the infinite Euler product for $z=2\varepsilon+1$ and $N=n-2$
$$\Gamma(z)= \lim_{n\to \infty} \frac{N!(N+1)^z}{z(z+1)\cdots(z+N)}$$

\begin{acknowledgment}
    Many thanks to T. Creutzig for many helpful discussions and private hospitality, as well as to G. Dhillon and T. Dimofte for many helpful discussions. 
    
    BF thanks the Humboldt foundation for the Humboldt research prize, much work on this article was done during the respective visit in Hamburg, and we thank in this context C.~Schweigert, J. Teschner and C. Reiher. SL was supported by the Feyodor-Lynen fellowship during part of this project, and for the finishing touches SL thanks T. Arakawa at RIMS, and the MATRIX workshop ''Tensor Categories, Quantum Symmetries, and Mathematical Physics'' for hospitality. 
\end{acknowledgment}

\end{document}